\documentclass{amsart}
\usepackage[T1]{fontenc}
\usepackage[utf8]{inputenc}

\usepackage{amssymb}
\usepackage{amsthm}
\usepackage{graphicx}
\usepackage{import}
\usepackage{amsmath}
\usepackage{amscd}
\usepackage{mathtools}
\usepackage{verbatim}
\usepackage[foot]{amsaddr}

\usepackage[margin=1.2in]{geometry}
\usepackage{tikz-cd}

\newtheorem{theorem}{Theorem}[section]
\newtheorem{thmx}{Theorem}

\newtheorem{proposition}[theorem]{Proposition}
\newtheorem{lemma}[theorem]{Lemma}

\theoremstyle{remark}
\newtheorem*{remark}{Remark}

\makeatletter

\renewcommand{\paragraph}{%
	  \@startsection{paragraph}{4}%
	  {\z@}{1ex \@plus 1ex \@minus .2ex}{-1em}%
	  {\normalfont\normalsize\bfseries}%
}

\renewcommand{\P}{\mathbb{P}}
\renewcommand{\O}{\mathcal{O}}
\newcommand{\OP}{\mathcal{O}_{\P^1}}
\newcommand{\OC}{\mathcal{O}_{C}}
\newcommand{\mE}{\mathcal{E}}
\newcommand{\mL}{\mathcal{L}}

\newcommand{\mF}{\mathcal{F}}
\newcommand{\mP}{\mathcal{P}}
\renewcommand{\S}{\mathcal{S}}
\renewcommand{\i}{\iota}

\newcommand{\BunO}{\Bun_{\O}} 
\newcommand{\Buninf}{\Bun_{\wi}}
\newcommand{\Bun}{\textnormal{Bun}}
\newcommand{\BunOm}{\BunO^{<}}
\newcommand{\BunOM}{\BunO^{>}}
\newcommand{\Buninfm}{\Buninf^{<}}
\newcommand{\BuninfM}{\Buninf^{>}}
\newcommand{\Bundemi}{\Bun_0^{\ubmu}(\P^1, \uD)}
\newcommand{\BunOCT}{\BunO^{\bmu}(C,T)}

\newcommand{\wi}{w_{\infty}}
\newcommand{\PP}{\P^1_z \times \P^1_w}
\newcommand{\PPsin}{\P^1 \times \P^1}
\newcommand{\Pb}{\P^2_{\bb}}

\renewcommand{\l}{\boldsymbol{l}}
\newcommand{\bmu}{\boldsymbol{\mu}}
\newcommand{\ubmu}{\underline{\bmu}}
\newcommand{\bnu}{\boldsymbol{\nu}}
\newcommand{\bl}{\boldsymbol{l}}
\newcommand{\bm}{\boldsymbol{m}}
\newcommand{\bb}{\boldsymbol{b}}
\newcommand{\bn}{\boldsymbol{n}}
\newcommand{\bp}{\boldsymbol{p}}

\newcommand{\T}{\mathcal{T}}

\newcommand{\uW}{\underline{W}}
\newcommand{\uD}{\underline{D}}

\newcommand{\Aut}{\textnormal{Aut}}

\newcommand{\AutS}{\textnormal{Aut}(\S)}

\newcommand{\PE}{\mathbb{P} E}
\newcommand{\Hom}{\text{Hom}}
\newcommand{\JacC}{\text{Jac}(C)}

\newcommand{\Forget}{\textnormal{Forget}}
\newcommand{\Tu}{\textnormal{Tu}}
\newcommand{\tPhi}{\tilde{\Phi}}
\newcommand{\tOmega}{\tilde{\Omega}}
\newcommand{\tsigmak}{\tilde{\sigma_k}}
\newcommand{\Elem}{\textnormal{Elm}}
\newcommand{\elem}{\textnormal{elm}}

\title{Geometry of the moduli of parabolic bundles \\ on elliptic curves}
\author{Néstor Fernández Vargas}
\address{IRMAR\\ Univ. Rennes 1\\ F-35000 Rennes\\ France}
\email{nestor.fernandez-vargas@univ-rennes1.fr}
\thanks{The author gratefully acknowledges support by the Centre Henri Lebesgue (ANR-11-LABX-0020-01).}


\begin{document}
\begin{abstract}
  The goal of this paper is the study of simple rank 2 parabolic vector bundles over a $2$-punctured elliptic curve $C$. We show that the moduli space of these bundles is a non-separated gluing of two charts isomorphic to $\PPsin$. We also showcase a special curve $\Gamma$ isomorphic to $C$ embedded in this space, and this way we prove a Torelli theorem. 
  This moduli space is related to the moduli space of semistable parabolic bundles over $\P^1$ via a modular map which turns out to be the 2:1 cover ramified in $\Gamma$. We recover the geometry of del Pezzo surfaces of degree 4 and we reconstruct all their automorphisms via elementary transformations of parabolic vector bundles.
\end{abstract}

\maketitle

\setcounter{tocdepth}{1}

\tableofcontents

\section{Introduction} 
  Let $X$ be a smooth projective curve over $\mathbb{C}$ and $T = t_1 + \cdots + t_n$ a reduced divisor on $X$. A rank 2 parabolic bundle $(E,\bp = (p_1,\ldots,p_n))$ over $(X, T)$ consists in a rank 2 vector bundle $E$ over $X$ together with 1-dimensional linear subspaces $p_i$ of the fiber  $E_{t_i}$ of $t_i$ for $i = 1,\ldots,n$.
  If $X$ is hyperelliptic, we can study the relation between parabolic bundles over $X$ and over its hyperelliptic quotient.
  The case $g=0$ is explored in \cite{lorayisom}, and the case $g=2$, building on \cite{bolognesikummer} and \cite{bolognesiweddle}, is developed in \cite{heuloray}. Here we investigate the elliptic case. 

  We are interested in three notions associated to these objects: indecomposability, simplicity and stability. A parabolic bundle is indecomposable if it cannot be written as a sum of two parabolic line bundles, and simple if it does not admit non-scalar automorphisms (preserving parabolic directions). 
  Let $\bmu = (\mu_1, \ldots, \mu_n) \in [0,1]^n$ be a vector of weights. By definition, a parabolic bundle $\mE = (E,\bp)$ is $\bmu$-semistable if, for every line subbundle $L \subset E$, we have
\begin{align*}
  \text{ind}_{\bmu}(\mE, L) :=  \text{deg} E - 2 \text{deg} L + \sum_{p_i \not \subset L} \mu_i - \sum_{p_i \subset L} \mu_i \geq 0.
\end{align*}    
  We say that the bundle $\mE$ is $\bmu$-stable when the inequality is strict for every $L$.
  The scalar $\text{ind}_{\bmu}(\mE, L)$ is the $\bmu$-parabolic degree of $(E, \bp)$.  

  These three notions coincide when the curve is the projective line (see \cite{loraysaito}), but they are different for curves of higher genus. In this paper, we study the case of an elliptic curve $C$ with a divisor $T = t_1 + t_2$ of degree 2. We show that, in this situation, $\bmu$-stability implies simplicity and simplicity implies indecomposability: 
\begin{align*}
  \bmu \text{-stable} \implies \text{simple} \implies  \text{indecomposable}. 
\end{align*}
  We also show that every simple bundle is $\bmu$-stable for some $\bmu$. Thus, the set of simple parabolic bundles is the union of the different sets of $\bmu$-stable bundles when $\bmu$ varies. In Section \ref{sec:moduli} we describe explicitely the moduli spaces of $\bmu$-semistable bundles. We prove the following Theorem:
\begin{thmx} \label{th:p1xp1}
  The moduli space $\BunOCT$ of $\bmu$-semistable parabolic bundles with fixed determinant is isomorphic to $\PPsin$. 
\end{thmx}
  For $\mu_1 \not = \mu_2$, every $\bmu$-semistable bundle in this moduli space is $\bmu$-stable. For $\mu_1 = \mu_2$, we have a strictly $\bmu$-semistable locus $\Gamma \subset \BunOCT$, which is a bidegree $(2,2)$ elliptic curve in $\PPsin$. We show that this locus contains the initial data $(C,T)$, equivalent to the data of a degree 4 reduced divisor $\uD$ and a point $t \not \in \uD$ in $\P^1$. This is proved in the following Theorem of Torelli type, where we consider all curves smooth:
\begin{thmx} \label{thx:Torelli}
  The curve $\Gamma$ is isomorphic to $C$, and we can recover the divisor $T$ from the embedding $\Gamma \subset \PPsin$. Let $\mathbf{G} = \text{PGL} (2) \times \text{PGL} (2)$. There exist one-to-one correspondences between the following sets:
\begin{align*}
\left\{\parbox[p]{6em}{\begin{center} $(2,2)$-curves   $\Gamma \subset \PPsin$\end{center}}\right\} \Big/ \mathbf{G} \
\xleftrightarrow{\ 1:1 \ }
\left\{\parbox[p]{7em}{\begin{center} $2$-punctured elliptic curves  $(C, T)$ \end{center}}\right\} \Big/\sim & \
\xleftrightarrow{\ 1:1\ }
\left\{\parbox[p]{7em}{\begin{center} $4+1$-punctured rational curves  $(\P^1, \uD + t)$\end{center}}\right\} \Big/\sim.
\end{align*}
\end{thmx}

  The space of weights $\bmu$ is divided in two chambers $C_<$ and $C_>$ with associated moduli spaces $\BunOm(C,T)$ and $\BunOM(C,T)$. Inside each chamber, the points in $\BunOCT$ represent the same bundles, regardless of the weight. Along the curve $\Gamma$ occurs a \emph{wall-crossing phenomena}: a point in $\Gamma$ represents a different bundle for $\bmu$ in $C_<$ or $C_>$. 
  In contrast, the bundles represented in the spaces $\BunOm(C,T) \setminus \Gamma$ and $\BunOM(C,T) \setminus \Gamma$ are the same. Identifying identical bundles, we construct the non-separated scheme 
\begin{align*}
  \BunO (C,T)= {X_< \coprod X_>} \Big/ \sim
\end{align*}
which parametrizes the set of simple parabolic bundles on $(C,T)$.

  Let $\uD$ be a reduced divisor on $\P^1$ of degree 5. In ~\cite{loraysaito}, the authors construct the full coarse moduli space $\Bun_{-1}(\P^1, \uD)$ of rank 2 indecomposable parabolic vector bundles over $(\P^1, \uD)$ as a patching of projective charts. These charts are moduli spaces $\Bun^{\bnu}_{-1}(\P^1, \uD)$ of $\bnu$-semistable bundles for specific weights $\bnu$. Thus, when we move weights from one chamber to another, we change between the charts.

  One of these charts is isomorphic to the projective plane $\P^2$. We have five points in this chart, namely four points $D_i$ for $i = 0, 1, \lambda, t$; and a special point $D_t$. Moving weights, we change to the chart $\S$, which is the blow-up of $\P^2$ in these five points. This chart is by definition a del Pezzo surface of degree 4 (recall that a del Pezzo surface of degree $d$ is a blow-up of the plane in $9-d$ points). 

  We construct a 2:1 modular map $$\Phi: \S \cong \Bundemi \to \BunOCT \cong \PPsin.$$ This map is the 2:1 cover of $\PPsin$ ramified over the curve $\Gamma$.
  We give a modular interpretation of the geometry of this covering. We start by studying the automorphism $\tau$ of $\S$ induced by $\Phi$. This map is the lift of a de Jonquières automorphism of the projective plane. More precisely, it is the lift of a birational transformation of $\P^2$ preserving the pencil of lines passing through $D_t$ and the pencil of conics passing through the four points $D_i$. Then, we show that the group $\Aut(\PPsin,\Gamma)$ of automorphisms of $\PPsin$ preserving $\Gamma$ is generated by three involutions $\tilde{\sigma_0}$, $\tilde{\sigma_1}$ and $\phi_T$ commuting pairwise. These involutions are modular, in the sense that correspond to natural parabolic bundle transformations. They lift to automorphisms $\sigma_0$, $\sigma_1$ and $\psi_T$ of $\S$.

These involutions allow us to reinterpret the group $\AutS$ of automorphisms of $\S$ in modular terms:
\begin{thmx} \label{th:autS}
  The group $\AutS$ is generated by the involutions $\sigma_k$, $\psi_T$ and $\tau$.
\end{thmx}

  These automorphisms are completely characterized by their action on a set $\Omega$, which is the union of the following geometric elements in $\P^2$: the five points $D_i$, the ten lines $\Pi_{ij}$ passing by $D_i$ and $D_j$ and the conic $\Pi$ passing by all the points $D_i$.
  We compute explicitely this action and the coordinate expression for each involution.

\paragraph{Acknowledgements} I thank my PhD advisors Frank Loray and Michele Bolognesi for their constant support and guidance. 
I am also very grateful to Thiago Fassarella for many interesting and fruitful conversations, and to Jérémy Blanc for valuable discussions on the subject.

\section{GIT moduli spaces and elementary transformations} \label{sec:par}
  In this Section, we recall the definitions of the main objects and maps that appear in this paper. These are moduli spaces of rank 2 parabolic vector bundles over $(X,T)$ and elementary transformations of parabolic vector bundles.

\subsection{The GIT moduli space of parabolic vector bundles}

  The coarse moduli space $\Bun^{\bmu}_L(X,T)$ of $\bmu$-semistable parabolic bundles of determinant $L$ over $(X,T)$ is a separated projective variety. The subset of $\bmu$-stable points is Zariski open and smooth, and the boundary is the $\bmu$-strictly semistable locus, containing the singular part (see \cite{mehtaseshadri}, \cite{bodenhu}, \cite{maruyamayokogawa}, \cite{seshadri}, \cite{kumar} and \cite{biswashollakumar}).

  Let $(E, \bp)$ be a rank 2 parabolic bundle over $(X,T)$ and $T' \subset T$ a subdivisor. A parabolic line subbundle of $(E, \bp)$ is a parabolic line bundle $(L, \bp')$ over $(X,T')$ such that $L \subset E$ and $L$ does not pass through the parabolic directions supported by $T'' = T - T'$
  (here, $\bp'$ is the unique parabolic structure over $(X, T')$).
  The quotient parabolic bundle is the parabolic line bundle $(E/L, \bp'')$ over $(X, T'')$. 
  The slopes of $\mL = (L,\bp')$ and $\mE = (E, \bp)$ are respectively the quantities
\begin{align*}
  \text{slope}(\mL) := \text{deg } L + \sum_{t_i \in T'} \mu_i \quad , \quad
  \text{slope}(\mE) := \frac{1}{2} \left(\text{deg } E + \sum_{t_i \in T} \mu_i \right)
\end{align*}
  With this notation, we have that $$\text{ind}_{\bmu}(\mE, L) = 2 \ \text{slope}(\mE) - 2 \ \text{slope}(\mL).$$ In particular,  $\mE$ is $\bmu$-semistable if, and only if, $\text{slope}(\mE) \geq \text{slope}(\mL)$ for every parabolic line subbundle $\mL$, the strict inequality corresponding to $\bmu$-stability.

  Let $\mE$ be $\bmu$-semistable of rank 2. A Jordan-Hölder filtration is a sequence
\begin{align*}
  0 \subset  \mL \subset \mE
\end{align*}
such that $\text{slope}(\mL) = \text{slope}(\mE)$. This sequence is not canonical, but the graded bundle
$  \text{gr} \mE = \mL \oplus \left( \mE / \mL \right)$
is canonical. The moduli space $\Bun_L^{\bmu}(X,T)$ is constructed by identifying bundles $\mE$ and $\mE'$ if $\text{gr} \mE \cong \text{gr} \mE'$ as holomorphic parabolic bundles. We say then that $\mE$ and $\mE'$ are in the same $s$-equivalence class (see \cite{mehtaseshadri} and \cite{bodenhu}).

  From now on, we will write $\Bun^{\bmu}_{B}(X,T)$ instead of $\Bun^{\bmu}_{\O(B)}(X,T)$ for a divisor $B$. 

\subsection{Elementary transformations} \label{subsec:elem}

  We recall the fundamental properties of elementary transformations of parabolic vector bundles. For a more complete reference, see \cite{machu}, \cite{heuloray}, \cite{lorayisom} or \cite{esnault}. 

  Let $\P (E, \bp)$ be the projectivization of the parabolic bundle $(E, \bp)$. It consist of the projective bundle $\PE$ together with a parabolic point $p_i$ in the fiber $F$ of each $t_i$. The \emph{elementary transformation} $\elem_{t_i}$ of $\P (E, \bp)$ is a birational transformation of the total space $\text{tot}(\PE)$: it is the blow-up of the point $p_i \in \PE$ followed by the contraction of the total transform $\hat{F}$ of the fibre $F$. The point resulting from this contraction gives the new parabolic direction $p_i'$. 
  
  In the vectorial setting $(E, \bp)$, we have two transformations which coincide projectively with the above definition: the positive elementary transformation $\elem^+_{t_i}$ and the negative elementary transformation $\elem^-_{t_i}$. We recall their properties in the following Proposition:

\begin{proposition}
  Let $(E, \bp)$ be a parabolic bundle over $(X, T)$. Then, the parabolic bundle $(E', \bp') = \elem^+_{t_i}(E, \bp)$ satisfies the following properties:
\begin{itemize}
  \item $\det (E', \bp') = \det (E, \bp) \otimes \O(t_i)$.
  \item If $L \subset E$ is a line subbundle passing by $p_i$, its image by $\elem^+_{t_i}$ is a subbundle $L' \cong L \otimes \O(t_i)$ of $\elem^+_{t_i}(E)$ not passing by $p_i'$. 
  \item If $L \subset E$ is a line subbundle not passing by $p_i$, its image by $\elem^+_{t_i}$ is a subbundle $L' \cong L$ of $\elem^+_{t_i}(E)$ passing by $p_i'$. 
\end{itemize}
  For the negative elementary transformation, the parabolic bundle $(E'', \bp'') = \elem^-_{t_i}(E, \bp)$ satisfies:
\begin{itemize}
  \item $\det (E'', \bp'') = \det (E, \bp) \otimes \O(-t_i)$.
  \item If $L \subset E$ is a line subbundle passing by $p_i$, its image by $\elem^-_{t_i}$ is a subbundle $L' \cong L $ of $\elem^+_{t_i}(E)$ not passing by $p_i''$. 
  \item If $L \subset E$ is a line subbundle not passing by $p_i$, its image by $\elem^-_{t_i}$ is a subbundle $L' \cong L \otimes (-t_i)$ of $\elem^-_{t_i}(E)$ passing by $p_i''$. 
\end{itemize}
\end{proposition}  

  From this Proposition, we obtain
\begin{align*}
	(E,\bp) \text{ is } (\bmu)\text{-semistable} \implies \elem^+_{t_i}(E,\bp) \text{ is } (\hat{\bmu}_i)\text{-semistable}.
\end{align*} 
where $\bmu = (\mu_1,  \ldots , \mu_n)$ and $\hat{\bmu}_i = (\mu_1,  \ldots, 1 - \mu_i , \ldots, \mu_n)$.
  Therefore, an elementary transformation is a class of isomorphisms, well-defined between the corresponding moduli spaces:
\begin{equation*}
	\Bun_L^{\bmu}(X,T) \xrightarrow{\elem_{t_i}^+} \Bun_{L(t_i)}^{{\hat{\bmu}}_1}(X,T).
\end{equation*}
  The inverse map is $\elem^-_{t_i}$, and the composition $\elem^+_{t_i} \circ \elem^+_{t_i}$ is the tensor product by $\O(t_i)$.
  If $t_i \not = t_j$, $\elem^*_{t_i} \circ \elem^*_{t_j} = \elem^*_{t_j} \circ \elem^*_{t_i}$.
  We denote by  $\elem^*_{T'}$ the composition $\elem^*_{T'} := \elem^*_{t_{i_1}} \circ \cdots \circ \elem^*_{t_{i_m}},$ 
  where $T' = t_{i_1} + \cdots + t_{i_m} \subset T$ is a subdivisor. 

  Tensoring  by a line bundle $M$ gives the \emph{twist automorphism}
\begin{equation*}
  \Bun^{\bmu}_L(X,T) \xrightarrow{\otimes M} \Bun^{\bmu}_{L \otimes M^2} (X,T).
\end{equation*}
  Since line bundles of even degree are always of the form $M^2$ for a suitable $M$, the moduli spaces $\Bun^{\bmu}_L(X,T)$ and $\Bun^{\bmu}_{L'}(X,T)$ are isomorphic if the degrees of $L$ and $L'$ have the same parity. Thus, it is enough to study the cases $L = \O$ and $L = \O(\wi)$, of moduli spaces $\BunO^{\bmu}(X,T)$ and $\Buninf^{\bmu}(X,T)$ respectively.

\section{Rank 2 indecomposable vector bundles over elliptic curves} \label{sec:rank}

  Let $X = C$ be the elliptic curve defined by the equation $y^2 = x (x - 1) (x - \lambda)$ and $\pi : C \to \P ^1$ be the elliptic cover $(x,y) \mapsto x$. Let $w_i$ be the four preimages of the points $i = 0, 1, \lambda, \infty$ via $\pi$.
  
\subsection{The 2-torsion group} \label{sec:torsion}
  Every degree 0 line bundle on $C$ is of the form $L = \O(p - \wi)$. Among these, we have the four torsion bundles $$L_k = \O(w_k - \wi)$$ corresponding to the Weierstrass points $w_k$ on $C$, for $k \in \Delta = \{0,1,\lambda,\infty\}$. The torsion bundles satisfy $L_k^2 = \O$. These bundles constitute the \emph{2-torsion group} $\mathcal{T}$ of $\JacC$.
 
 Let $q \in C$ be a point. We define an involution $\i_q: C \to C$ as follows: for every $p \in C$, $\i_q (p)$ is the unique point in $C$ satisfying the following equation of linear equivalence:
\begin{equation*}
p + q + \i_q(p) \sim 3 \wi.
\end{equation*}
  The group $\T$ acts on $\JacC$ by tensor product of line bundles. Each element $L_k$ of $\T$ thus induces an involution on $\JacC$, and therefore also on $C$. We have the following description:

\begin{proposition} \label{prop:beta}
  Let $k \in \Delta = \{0,1,\lambda,\infty\}$. The involution $C \xrightarrow{\otimes L_k} C$ coincides with the composition $\i_{w_k} \circ \i_{\wi}$. There exists an automorphism $\beta_k$ of $\P^1$ such that the following diagram commutes:
\begin{equation*}
\begin{tikzcd}
 C \arrow{r}{\otimes L_k} \arrow{d}{\pi} & C \arrow{d}{\pi}  \\
 \P^1 \arrow{r}{\beta_k}  & \P^1
\end{tikzcd}
\end{equation*}
  If $k = \infty$, $\beta_k$ is the identity map. 
  For $k \in \Delta \setminus \{\infty\}$, $\beta_k$ is the unique automorphism of $\P^1$ permuting the points in the pairs $\{k,\infty\}$ and $\Delta \setminus \{k , \infty \}$.
More precisely, we have:
\begin{align*}
  \beta_0 (z)          = \frac{\lambda}{z} ,\quad
  \beta_1 (z)          = \frac{z-\lambda}{z-1} ,\quad
  \beta_\lambda (z)    = \frac{\lambda z - \lambda}{z - \lambda}
\end{align*}
\
\end{proposition}
\begin{proof}
	Let $M = \O(p - \wi)$ be an arbitrary line bundle of degree 0. Then $M \otimes L_k = \O(p - w_k) = \O(q - \wi)$ for a suitable $q \in C$. The group law on $C$ implies that $q$, $w_k$ and $\i_{\wi}(p)$ are collinear in the projective model $\P^2$ of affine coordinates $(x,y)$. This implies $q = \i_{w_k} \circ \i_{\wi} (p)$.
  Since the hyperelliptic involution $\i_{\wi}$ and the twist $\i_{w_k} \circ \i_{\wi}$ commute for all $k$, the map $\beta_k$ is well defined.

  The set $W := \{w_0,w_1,w_\lambda,\wi\}$ is invariant under the map $\i_{w_k} \circ \i_{\wi}$. Since points $w_k$ in $C$ are projected by $\pi$ to points $k$ in $\P^1$ for $k \in \Delta$, the map $\beta_k$ is as described.
\end{proof}

\subsection{Indecomposable rank 2 vector bundles over $C$}
  The classification of vector bundles $E$ of rank 2 over $C$ was achieved by Atiyah in ~\cite{atiyahvectorbundles}. In this Section we recall some of his results. For a fixed determinant, the set of these bundles is parametrized by the Jacobian of $C$. 

  Let $E_0$, $E_1$ be the unique non trivial extensions given by exact sequences
\begin{equation*} \label{eq:E1}
	0 \to \O \to E_1 \to \O(\wi) \to 0 \qquad \text{and} \qquad 0 \to \O \to E_0 \to \O \to 0.
\end{equation*}
  We have the following result:

\begin{theorem}[Atiyah, \cite{atiyahvectorbundles}]\label{th:E1}
  The bundles $E_1$ and $E_0$ satisfy the following properties:
\begin{enumerate}
  \item The vector bundle $E_1$ is the unique indecomposable rank 2 bundle of determinant $\O(\wi)$ up to isomorphism. 
  \item Let $L \in \JacC$. Then, there is a unique inclusion $L \subset E_1$.
  \item The bundle $E_1$ does not admit non-scalar automorphisms.
  \item The indecomposable rank 2 vector bundles of trivial determinant are exactly those of the form $E_0 \otimes L_k$, where $L_k$ is a torsion line bundle.
  \item There is a unique inclusion $\O \subset E_0$.
  \item Let $T = t_1 + t_2$ be a reduced divisor in $C$. Consider the following set of parabolic bundles over $(C,T)$:
\begin{align*}
  P_1 = \{(E_0, \bp = (p_1, p_2)) \ | \ p_1 \subset \O \text{ and } p_2 \not \subset \O\}
\end{align*}
  The group of automorphisms of $E_0$ modulo scalar automorphisms acts transitively and freely on this set.  
\end{enumerate}
\end{theorem}
\begin{remark}
  By (3), $E_1$ only admits scalar automorphisms. On the other hand, multiplication by torsion line bundles give the automorphisms of the projectivisation bundle $\P E_1$.
  In contrast, $E_0$ admits non-scalar automorphisms. In fact, for every line bundle $L$ of degree $-1$ and every parabolic bundle in $P_1$, there is a unique inclusion $L \subset E_1$ such that $L$ passes by both parabolic directions $p_1$ and $p_2$.
\end{remark}

\paragraph{The geometry of $\P E_1$}
  Let us fix a point $p \in C$ and a line bundle $L \in \JacC$. By (3) in Theorem \ref{th:E1}, $L$ is a subbundle of $E_1$ and the inclusion is unique (up to scalar multiplication). This inclusion thus defines a single parabolic direction $m_p(L)$ in the fiber $E_1|_p$ of $p$. Hence, we have a map 
\begin{align*}
  m_p : \JacC \to  E_1|_p \cong \P^1_z.
\end{align*}
  The image $m_p(L)$ can be seen as a point in the fiber $\P E_1 |_p$ of the projective bundle $\P E_1$. The line subbundle $L \subset E_1$  correspond to a cross-section $s_L$ of self-intersection $+1$ of this ruled surface.
  The following Proposition describes the geometrical situation of these cross-sections: for every point $m$ in the fiber $\P E_1 |_p$ there are generically two +1 cross-sections of $\P E_1$ passing through $m$. 
\begin{proposition} \label{prop:2sec}
  The map $m_p: \JacC \to \P^1_{z}$ is 2:1 ramified in four points. If $z \in \P^1_z$ is a non-ramification point, we have
\begin{align*} 
  m_p^{-1}(z) = \{\O(q - w_\infty) , \O(\i_{p}(q) - w_\infty) \}
\end{align*}
for $q \in C$.
  The preimages of ramification points are the four line bundles $M = \O(q - w_\infty)$ with $\i_p(q) = q$. 
  They satisfy $M^2 = \O(\wi - p)$.
\end{proposition}
\begin{proof}
  Let $L = \O(q - \wi) \in \JacC$ be a non-ramification point of $m_p$. Let  $L' = \O(\tilde{q} - \wi) \not = L$ such that $m_p(L) = m_p(L') = m$. The associated cross-sections $s_L$ and $s_{L'}$ only intersect in $m$. In particular, the underlying bundle $E$ of the image of the parabolic bundle $(E_1, m_p(L))$ over $(C,p)$ via $\elem^+_p$ is decomposable:
\begin{align*}
  E = \O(p+q-\wi) \oplus \O(p + \tilde{q} - \wi)
\end{align*}
  By the properties of elementary transformations, we have that $\det E = \O(\wi + p)$, which implies $\tilde{q} = \i_p(q)$. 
  Suppose that there is a third distinct line bundle $L''$ passing through $m$. Then, the projective bundle $\P E$ has three non crossing $+0$ sections $s_L$, $s_{L'}$, $s_{L''}$, implying that $E$ is trivial. But then $E_1 = \elem^-_p(E)$ is decomposable, which is absurd. This proves that the map $m_p$ is 2:1, and the four ramification points are the fixed points of the involution $\i_p$.
\end{proof}

\subsection{The Tu isomorphism} \label{sec:tu}
  We say that a rank 2 (non parabolic) vector bundle $E$ over $C$ is semistable (resp. stable) if it is $\bmu$-semistable (resp. stable) for $\bmu = (0, \ldots, 0)$. The moduli space $\BunO(C)$ of semistable rank 2 vector bundles over $C$ with trivial determinant is constructed in \cite{tu}.
 
  The bundles represented in $\BunO(C)$ are in fact all strictly semistable. They are the decomposable bundles $L \oplus L^{-1}$, together with the indecomposable bundles $E_0 \otimes L_k$, with $L_k$ a torsion line bundle.

\begin{theorem}[Tu, \cite{tu}] \label{prop:Tu}
  The moduli space $\BunO(C)$ is isomorphic to the quotient of $\JacC$ by the hyperelliptic involution. More precisely, there is a canonical isomorphism 
$$\BunO(C) \xrightarrow{\Tu} \P^1$$
such that $\Tu (\mP) = \pi (p)$, where $\mP = [L \oplus L^{-1}]$ for $L = \O (p - \wi)$ and $\pi$ is the hyperelliptic cover. Moreover, we have
\begin{itemize}
  \item If $p \in \P^1 \setminus \{0,1,\lambda,\infty\}$, $\mP$ represents a single isomorphism class $[L \oplus L^{-1}]$.
  \item If $p = k \in \{0,1,\lambda,\infty\}$, $\mP$ represents two different isomorphism classes $[L_k \oplus L_k]$ and $[E_0 \otimes L_k]$.
\end{itemize}
\end{theorem}

  The forgetful map 
  $\BunOCT \xrightarrow{\Forget} \BunO(C)$ 
is defined in the obvious way: $$\Forget[E,\bp] = [E].$$

\section{The coarse moduli space of parabolic bundles over an elliptic curve} \label{sec:moduli}

  Let $T = t_1 + t_2$ be a reduced divisor on $C$. We can assume that the 2:1 cover $\pi:C \to \P^1$ satisfies $\pi(t_1) = \pi(t_2) = t \in \P^1$.
  In this Section, we describe the GIT moduli space $\Bun^{\bmu}_L(C,T)$ of $\bmu$-semistable rank 2 parabolic bundles with parabolic directions over $T$ and fixed determinant $L$.
 This moduli space depends on the choice of weights. More precisely, there is a hyperplane cutting out $[0,1]^2$ into two chambers, and strictly $\bmu$-semistable bundles only occur along this wall. The moduli space $\Bun^{\bmu}_L(C,T)$ is constant in each chamber, i. e. the set of parabolic bundles represented by each of its points does not vary.

  While we are mostly interested in the trivial determinant case, it turns out that the computations are easier in the odd degree case. Therefore, we will first study the odd degree case and then translate our results to the even degree setting.

\subsection{The odd degree moduli space} \label{sub:comp}

  In this Section, we describe the moduli space  $\Buninf^{\bmu}(C,T)$, giving a proof of Theorem \ref{th:p1xp1}. 
  Let us define the \emph{wall} $W \subset [0,1]^2$ as the hyperplane $\mu_1 + \mu_2 = 1$. We will see that semistable bundles arise only when weights are in $W$. Let us start by listing $\bmu$-semistable and $\bmu$-stable bundles in $W$:

\begin{proposition} \label{prop:clasif}
  Let $\bmu$ be weights in the wall $W$, with $\mu_i \not = 0$ for $i = 1,2$. Then, the parabolic bundles represented in $\Buninf^{\bmu}(C,T)$ are exactly the following:
\begin{itemize}
  \item $\bmu$-stable bundles: 
  \begin{itemize}
    \item  $(E_1, \bm)$ with parabolic directions $\bm$ not lying on the same subbundle $L \in \JacC$.
  \end{itemize}
  \item Strictly $\bmu$-semistable bundles: 
  \begin{itemize}
    \item $(E_1, \bm)$ with parabolic directions $\bm$ lying on the same subbundle $L \in \JacC$. 
    \item $E = L \oplus L^{-1}(\wi)$ with $L \in \JacC$  and no parabolics lying on $L^{-1}(\wi)$.
  \end{itemize}
\end{itemize}

   If $\mu_i = 0$, we also find the bundles $E = L \oplus L^{-1}(\wi)$ with $m_i$ lying on $L^{-1}(\wi)$.
\end{proposition}
\begin{proof}
  By Theorem \ref{th:E1}, the underlying vector bundle of an element of $\Buninf^{\bmu} (C,T)$ is either decomposable or $E_1$. The result follows from direct computation of the $\bmu$-parabolic degree of each bundle.
\end{proof}
  
  Now we describe $s$-equivalence classes. Let $\Gamma$ be the strictly $\bmu$-semistable locus. We have the following result:

\begin{theorem} \label{th:strict}
  Let $\bmu$ be weights in the wall $W$. Then, the locus $\Gamma \subset \Buninf^{\bmu}(C,T)$ is a  curve parametrized by the Jacobian of $C$. More precisely, if $0 < \mu_1, \mu_2 < 1$, for each point $\mL$ in $\Gamma$ there exists a unique $L \in \JacC$ such that $\mL$ represents the following three isomorphism classes of parabolic vector bundles:
\begin{itemize}
  \item $\mE_<(L) =(E_1, \bm)$ with both parabolic directions lying on $L \subset E_1$.
  \item $\mE_>(L) =(L \oplus L^{-1}(\wi), \bp)$ with parabolics outside of $L^{-1}(\wi)$, not on the same $L$.
  \item $\mE_=(L) =(L \oplus L^{-1}(\wi), \bp)$ with both parabolic directions on the same $L$.
\end{itemize}
\end{theorem}
\begin{proof}
  We claim that the three bundles above are all of the same $s$-equivalence class, i. e. they are all identified in a single point in the moduli space. Indeed, a Jordan-Hölder filtration for the first configuration is $0 \subset (L, \bp') \subset \mE_<(L)$, where $(L, \bp')$ is the unique parabolic structure over $(X,T)$. The bundles $(L, \bp')$ and $\mE_<(L)$ have equal slope 1. This gives the graded bundle 
\begin{equation*}
  \text{gr}_{\bmu} \mE_<(L) = (L, \bp') \oplus (L^{-1}(\wi),\phi).
\end{equation*}
  For the second configuration, we choose the filtration $0 \subset (L^{-1}(\wi), \phi) \subset \mE_>(L)$
with associated graded bundle 
\begin{equation*}
  \text{gr}_{\bmu} \mE_>(L) = (L^{-1}(\wi), \phi) \oplus (L,\bp).
\end{equation*}
  Since clearly $\text{gr}_{\bmu} \mE_<(L) \cong \text{gr}_{\bmu} \mE_>(L)$, the parabolic bundles $\mE_<(L)$ and $\mE_>(L)$ are in the same $s$-equivalence class. The second filtration works also for the bundle $\mE_=(L)$, hence this bundle is also identified with the previous two.

  From the description of strictly $\bmu$-semistable bundles in Proposition \ref{prop:clasif}, it is clear that no other bundle belongs to the same $s$-equivalence class. Consequently, the map $\Gamma \to \JacC$ given by $[\mE_*(L)] \mapsto L$ is 1:1.
\end{proof}
  
  Fix $\bmu$ in the wall $W$. By Proposition \ref{prop:clasif}, $\bmu$-stable bundles are of the form $(E_1, \bm)$. We identify $E_1|_{t_1} \times E_1|_{t_2} \cong \P^1_{m_1} \times \P^1_{m_2}$, where $E_1|_{t_i}$ is the fiber of $t_i$. Define
\begin{align*}
  M: \P^1_{m_1} \times \P^1_{m_2} \to \Buninf^{\bmu}(C,T) 
\end{align*}
by $M(m_1, m_2) = [E_1, \bm = (m_1, m_2)]$. 
  We are now in position to prove Theorem \ref{th:p1xp1} when weights are in the wall.
  
\begin{theorem} \label{th:p1xp1wall}
  Fix a vector of weights $\bmu$ in the wall $W$. Then, the map $M$ is an isomorphism. The curve $\Gamma$ is of bidegree (2,2) in $\P^1_{m_1} \times \P^1_{m_2}$.
\end{theorem}

\begin{proof}
	Let $\mE = (E_1, \bm)$ be a parabolic bundle. If $\mE$ is stable, then it is the only parabolic bundle in its $s$-equivalence class. If $\mE$ is semistable, Theorem \ref{th:strict} shows that it is the only parabolic bundle in its $s$-equivalence class having $E_1$ as underlying bundle. This shows injectivity, and surjectivity is also clear from Theorem \ref{th:strict}.

  Strict $\bmu$-semistability occurs when there is a degree 0 line subbundle $L$ passing through both parabolic directions. For a generic choice of $m_1$, there are two of these line subbundles passing through $m_1$, thus defining generically two parabolic directions on the fiber of $t_2$ (see Proposition \ref{prop:2sec}). Hence, the locus of strictly $\bmu$-semistable parabolic configurations is a curve of bidegree (2,2).
  Theorem \ref{th:strict} shows that this locus is exactly $\Gamma$.
\end{proof}

  When we move the weights $\bmu$ outside the wall and inside the chambers, the family of $\bmu$-semistable bundles changes. More precisely, some of the formerly strictly $\bmu$-semistable bundles become stable, while some others become unstable. Therefore, points in $\Gamma$ represent different isomorphism classes of bundles depending on the choice of weights. The following proposition summarizes the situation:

\begin{proposition} \label{prop:outsideodd}
  Let $\bmu$ be a vector of weights outside the wall. Then, the moduli space $\Buninf^{\bmu} (C,T)$ is still isomorphic to $\P_{m_1} \times \P_{m_2}$. Each point in this space represent a single $\bmu$-stable bundle. Points outside $\Gamma$ represent a single parabolic configuration $(E_1, \bm)$ with $\bm$ not lying on the same line subbundle $L \in \JacC$. 
  Each point $\mL \in \Gamma$ represents:
\begin{itemize}
  \item The bundle $\mE_<(L)$, if $\bmu \in I_<$.
  \item The bundle $\mE_>(L)$, if $\bmu \in I_>$.
\end{itemize}
\end{proposition}

\begin{proof}
  It is sufficient to show that the bundle $\mE_<(L)$ is $\bmu$-stable when $\mu_1 + \mu_2 < 1$ and $\bmu$-unstable when $\mu_1 + \mu_2 > 1$, and that the bundle $\mE_>(L)$ is respectively $\bmu$-stable and $\bmu$-unstable in the opposite chambers. This is proved by computing the parabolic $\bmu$-degree of $L$ in each case.
\end{proof}

  We have proven that $\Buninf^{\bmu}(C,T)$ is isomorphic to $\P^1 \times \P^1$ for every $\bmu$. The same result holds for $\BunO^{\bmu}(C,T)$ applying an elementary isomorphism. This completes the proof of Theorem \ref{th:p1xp1}.
  
  The moduli spaces $\Buninf^{\bmu}(C,T)$ and $\Buninf^{\bnu}(C,T)$ represent the same bundles if, and only if, both weights $\bmu$ and $\bnu$ belong to the same set among the following three:
\begin{itemize}
  \item The \emph{chamber} $I_< = \{(\mu_ 1, \mu_2) \in (0,1)^2 \ |\ \mu_ 1 + \mu_2 < 1\}$. 
  \item The \emph{chamber} $I_> = \{(\mu_ 1, \mu_2) \in (0,1)^2 \ |\ \mu_ 1 + \mu_2 > 1\}$.
  \item The \emph{wall} $W = \{(\mu_ 1, \mu_2) \in (0,1)^2 \ |\ \mu_ 1 + \mu_2 = 1\}$.
\end{itemize}
  
We will thus adopt the notations $\Buninfm(C,T)$, $\BuninfM(C,T)$ and $\Buninf^= (C,T)$ for the moduli spaces $\Buninf^{\bmu} (C,T)$ and $\bmu$ in $I_<$, $I_>$ and $W$ respectively.

\subsection{The even degree moduli space}

  In this Section we use the description of the odd degree moduli space to describe the moduli space $\BunOCT$. The map $$\Elem_{t_2}^+ := \elem_{t_2}^+ \otimes R$$ provides an isomorphism between $\Buninf^{\bmu}(C,T)$ and  $\BunO^{\bmu_2}(C,T)$, where $R$ is a convenient line bundle  and $\bmu_2 = (\mu_1, 1 - \mu_2)$ (see Table \ref{table:bundles}).
  The wall and chambers change in even degree.  The moduli spaces $\BunOCT$ and $\BunO^{\bnu}(C,T)$ represent the same bundles if, and only if, both weights $\bmu$ and $\bnu$ belong to the same set among the following three:
\begin{itemize}
  \item The \emph{chamber} $J_< = \{(\mu_ 1, \mu_2) \in (0,1)^2 \ |\ \mu_ 1 < \mu_2\}$.
  \item The \emph{chamber} $J_> = \{(\mu_ 1, \mu_2) \in (0,1)^2 \ |\ \mu_ 1 > \mu_2\}$.
  \item The \emph{wall} $W' = \{(\mu_ 1, \mu_2) \in (0,1)^2 \ |\ \mu_ 1 = \mu_2\}$.
\end{itemize}

  Recall that we have two kinds of semistable rank 2 vector bundles $E$ of trivial determinant: the decomposable case, i. e. where $E = L \oplus L^{-1}$ with $\text{deg} \ L = 0$, and the indecomposable case $E_0 \otimes L_k$ with $L_k$ a torsion line bundle.
  Let us start  by fixing weights $\bmu$ in the wall, as in the odd degree case. 
\begin{proposition} \label{prop:clasifpar}
  If $\bmu \not = 0$, the parabolic bundles represented in $\BunO^=(C,T)$ are the following:
\begin{itemize}
  \item $\bmu$-stable bundles:
  \begin{itemize}
    \item $(L \oplus L^{-1}, \bm)$, $L \not = L_k$, with no parabolics on $L$ nor on $L^{-1}$. 
    \item $(E_0 \otimes L_k, \bm)$ with no parabolics on $L_k \subset E_0 \otimes L_k$. 
  \end{itemize}
  \item Strictly $\bmu$-semistable bundles:
  \begin{itemize}
    \item $(L \oplus L^{-1}, \bm)$, $L \not = L_k$, with one or two parabolics on $L$ or $L^{-1}$, but not both on the same. 
    \item $(E_0 \otimes L_k, \bm)$ with exacly one parabolic on $L_k \subset E_0 \otimes L_k$. 
    \item $(L_k \oplus L_k, \bm)$ with parabolics not lying on the same $L_k$. 
  \end{itemize}
\end{itemize}
\end{proposition}

\begin{proof}
  By applying the elementary transformation $\Elem^+_{t_2}$ to the bundles of Proposition \ref{prop:clasif}, we obtain the described bundles.
  The different cases $L = L_k$ and $L \not = L_k$ correspond in the odd degree setting to $E_1$ having one or two degree 0 bundles passing through $m_2$.
\end{proof} 

  Let us keep the notation $\Gamma$ for the strictly $\bmu$-semistable locus in $\BunOCT$. This locus is a curve parametrized by $\JacC$ by Theorem \ref{th:p1xp1wall}, and the description of the bundles represented by the points in $\Gamma$ is as follows:
\begin{theorem} \label{th:strict2}
  Let $0 < \mu_1 = \mu_2 < 1$. For each point $\mL$ in $\Gamma \subset \BunOCT$ there exists a unique $L \in \JacC$ such that the point $\mL$ represents the following three isomorphism classes of parabolic vector bundles:
\begin{itemize}
  \item If $L = L_k$ is a torsion line bundle:
  \begin{itemize}
    \item $\mF_<(L_k) =(E_0 \otimes L_k,  \bm)$ with $m_1$ on $L_k \subset E_0 \otimes L_k$ and $m_2$ not on $L_k$.
    \item $\mF_>(L_k) =(E_0 \otimes L_k,  \bm)$ with $m_2$ on $L_k \subset E_0 \otimes L_k$ and $m_1$ not on $L_k$
    \item $\mF_=(L_k) =(L_k \oplus L_k , \bm)$ with parabolic directions not lying on the same $L_k$.
  \end{itemize}
  
  \item If $L$ is not a torsion line bundle:
  \begin{itemize}
    \item $\mF_<(L) =(L \oplus L^{-1} , \bm)$ with $m_1 \in L$ and $m_2$ out of $L$ and $L^{-1}$.
    \item $\mF_>(L) =(L \oplus L^{-1} , \bm)$ with $m_2 \in L^{-1}$ and $m_1$ out of $L$ and $L^{-1}$.
    \item $\mF_=(L) =(L \oplus L^{-1} , \bm)$ with $m_1 \in L$ and $m_2$ in $L^{-1}$.
  \end{itemize}

\end{itemize}
\end{theorem}

\begin{proof}
  One way to prove this result is to translate the situation of Theorem \ref{th:strict2}, computing the images of the bundles $\mE_<(L)$, $\mE_>(L)$ and $\mE_=(L)$ in $\Buninf^=(C,T)$ by the elementary transformation $\Elem^+_{t_2}$.
	Let $M \in \JacC$ be an arbitrary line bundle and $L = M(t_2 + R)$, with $2R = -t_1 - \wi$. We claim that
\begin{align*}
  \Elem^+_{t_2} (\mE_*(M)) = \mF_*(L)
\end{align*}
for $* \in \{<,>,=\}$.
  Let us prove the assertion for the first case. The bundle $\mE_<(M)=(E_1, \bm)$ has parabolic directions $\bm$ lying on $M \subset E_1$. If $L = L_k$, Proposition \ref{th:strict2} states that $M$ is the unique degree 0 subbundle of $E_1$ passing through $m_2$ . This implies that the underlying bundle of $\Elem^+_{t_2} (\mE_k(M))$ is indecomposable, and thus equal to $E_0 \otimes L_k$. 
  If $L \not = L_k$, the underlying bundle of $\Elem^+_{t_2} (\mE_k(M))$ is decomposable, therefore it is $L \oplus L^{-1}$. The position of parabolic directions is given by the configuration in $\mE_<(M)$. For $k > 1$, the proof is similar.
\end{proof}

    Alternatively, one can also show that strictly semistable configurations in the even degree case are given by Jordan-Hölder filtrations 
\begin{align*}
  0 \subset (L, m_1') \subset &\mF_*(L) \qquad \text{for } * \in \{<,>\} \\
  0 \subset (L^{-1}, m_2') \subset &\mF_=(L) 
\end{align*}
where $(L, m_1')$ (resp. $(L, m_2')$ is the line subbundle over $(C, t_1)$ (resp. $(C, t_2)$). These give isomorphic graded bundles. Therefore, $\mF_<(L)$, and $\mF_>(L)$ are identified in the moduli space.
  Moving weights outside the wall, we find the same situation as in the odd degree case: some of the strictly semistable bundles $\mF_*(L)$ become stable, while some become unstable. 

\begin{proposition} \label{prop:strict2}
  Let $\bmu$ be weights outside the wall and $[E, \bm] \in \BunOCT \setminus \Gamma$. Then, one of the following holds:
\begin{itemize}
  \item $E = L \oplus L^{-1}$, with $L \not = L_k$, and parabolic directions outside $L$ and $L^{-1}$.
  \item $E = E_0 \otimes L_k $, with parabolic directions outside $L_k$.
\end{itemize}   
  Each point $\mL \in \Gamma$ represents:
\begin{itemize}
    \item The bundle $\mF_<(L)$ if $\bmu \in J_<$.
    \item The bundle $\mF_>(L)$ if $\bmu \in J_>$.
\end{itemize}
\end{proposition}

  We follow the notation of the previous section: for $\bmu$ in $J_<$ (resp. in $J_>$), we will denote by $\BunO^< (C,T)$ (resp. $\BunO^> (C,T)$) the moduli space $\BunO^{\bmu} (C,T)$.
  
  \begin{table}[h]
\begin{center}
\begingroup%
  \makeatletter%
  \providecommand\color[2][]{%
    \errmessage{(Inkscape) Color is used for the text in Inkscape, but the package 'color.sty' is not loaded}%
    \renewcommand\color[2][]{}%
  }%
  \providecommand\transparent[1]{%
    \errmessage{(Inkscape) Transparency is used (non-zero) for the text in Inkscape, but the package 'transparent.sty' is not loaded}%
    \renewcommand\transparent[1]{}%
  }%
  \providecommand\rotatebox[2]{#2}%
  \ifx\svgwidth\undefined%
    \setlength{\unitlength}{439.75460815bp}%
    \ifx\svgscale\undefined%
      \relax%
    \else%
      \setlength{\unitlength}{\unitlength * \real{\svgscale}}%
    \fi%
  \else%
    \setlength{\unitlength}{\svgwidth}%
  \fi%
  \global\let\svgwidth\undefined%
  \global\let\svgscale\undefined%
  \makeatother%
  \begin{picture}(1,0.69166653)%
    \put(0,0){\includegraphics[width=\unitlength]{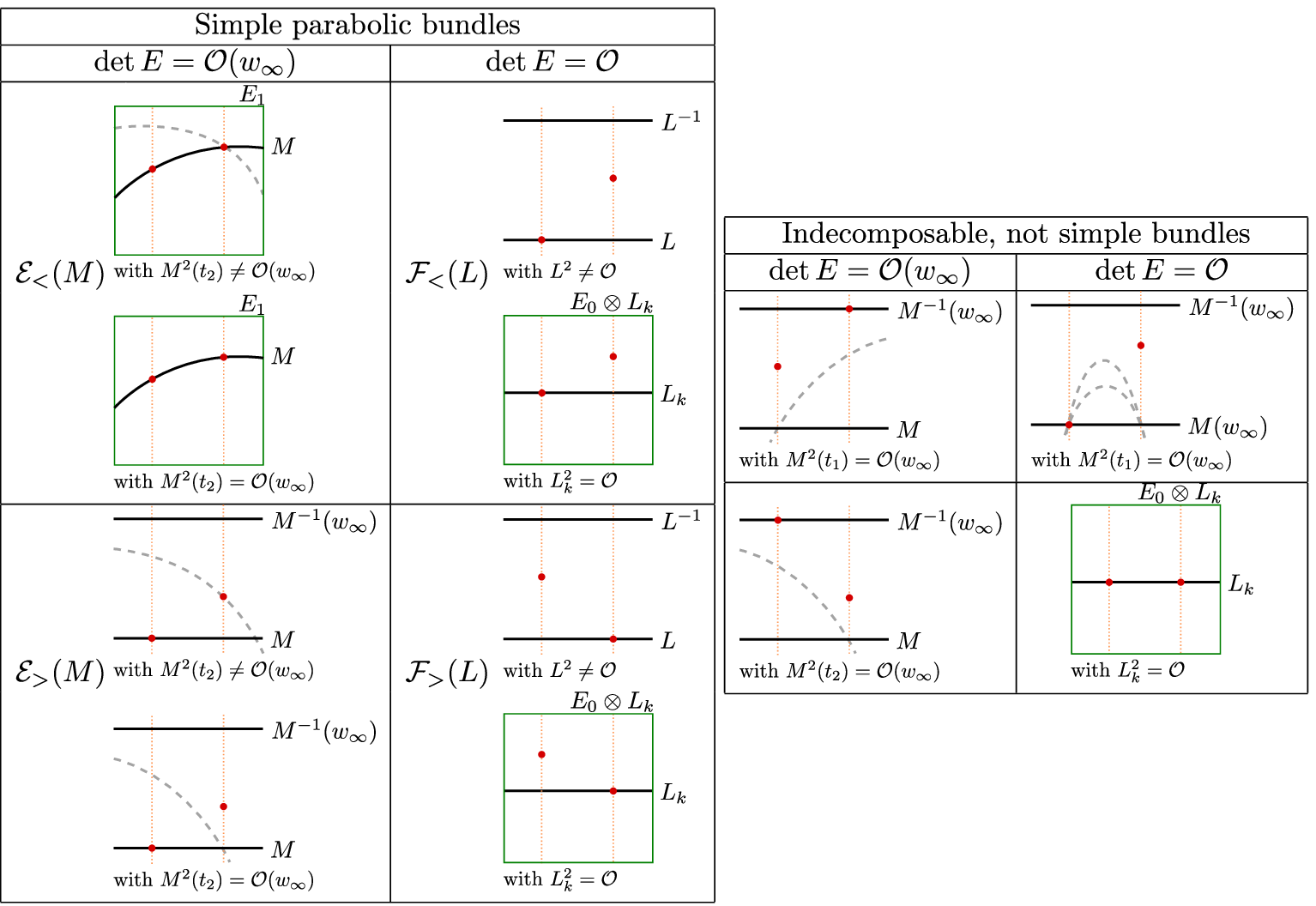}}%
  \end{picture}%
\endgroup%

\end{center}
\caption{Lists of simple and indecomposable parabolic bundles over $(C,T)$. Each figure represent the corresponding projective bundle, and each parabolic bundles on the right side a obtained by applying the elementary transformation $\Elem^+_{t_2}$ to the corresponding bundle on the left side.}
\label{table:bundles}
\end{table}

\subsection{Coordinate systems}
  In this Section we define coordinate systems for the moduli spaces $\Bun_L^=(C,T)$. These coordinate systems will be automatically defined on $\Bun_L^<(C,T)$ and $\Bun_L^>(C,T)$ since every bundle represented in these spaces is already in $\Bun_L^=(C,T)$.
  We will denote by $\P^1_z$ (resp. $\P^1_w$) the projective line with homogeneous coordinates $z = (z_0 : z_1)$ (resp. $w = (w_0 : w_1)$).
\paragraph{The coordinate system for $\BunO^=(C,T)$} \label{sub:coordinatepair}
  We start by defining an automorphism $\phi_T$ of $\BunO^=(C,T)$. 

  Let $\mP = [E, \l]$ be an element of $\BunO^=(C,T)$. 
  The bundle $\elem^+_T(E, \l)$ has determinant $\O(2 \wi)$. Define the bundle $\phi_T(E, \bm) := \elem^+_T(E, \l) \otimes \O(-\wi)$, with trivial determinant. This bundle is an element of $\BunO^=(C,T)$, and it does not depend on the representative bundle of $\mP$. We thus obtain a 1:1 automorphism
\begin{align*}
  \phi_T : \BunO^=(C,T) \to \BunO^=(C,T)
\end{align*}

  We define the coordinate system on $\BunO^= (C,T)$ as follows. Consider the moduli space $\BunO(C,T)$ defined in Section \ref{sec:tu}: by composing the map $\phi_T$ with the forgetful and Tu maps we obtain the diagram
\begin{equation*}
\begin{CD}
  \BunO^= (C,T)               @>\phi_T>>      \BunO^= (C,T)           \\
  @V\Forget VV                      @VV\Forget V  \\
  \BunO(C)             @.              \BunO(C)         \\
  @V{\Tu}V\cong V                     @V\cong V{\Tu}V \\
  \P ^1_z           @.              \P^1_w        \\
\end{CD}
\end{equation*}
  The \emph{coordinate map} in $\BunO^= (C,T)$ is the mapping
\begin{equation*}
  \BunO^= (C,T)  \xrightarrow{\varepsilon} \PP
\end{equation*}
defined by $\varepsilon = ({\Tu} \circ \Forget) \times ({\Tu} \circ \Forget \circ \phi_T)$. 
  We will prove the following Proposition using the coordinate system in the odd degree case:
\begin{proposition} \label{prop:1:1}
The mapping $\varepsilon$ is 1:1.
\end{proposition}
\paragraph{The coordinate system for $\Buninf^=(C,T)$} \label{sub:coordinateimpair}
  Recall that $T = t_1 + t_2$. Let us write $t_1 = (t, s)$ and $t_2 = (t, -s)$, where $C = \{y^2 = x (x - 1) (x - \lambda)\}$. 
  Consider the maps $\varepsilon_1  : C \to \P^1_{\tilde{z}}$ and $\varepsilon_2  : C \to \P^1_{\tilde{w}}$ defined by
\begin{equation*}
  \varepsilon_1 (p) = \frac{t y - s x}{y - s}, \
  \varepsilon_2 (p) = \frac{t y + s x}{y + s}.
\end{equation*}

  A straightforward calculation shows that, for $j = 1,2$, the maps $\varepsilon_j$ are 2:1 with preimages of the form $\{p, \i_{t_j}(p)\}$, where $\i$ is the involution defined in Section \ref{sec:torsion}. Moreover, $\varepsilon_i(w_k) = k$ for $k \in \{0, 1, \lambda\}$ since $w_k$ has coordinates $(0,k)$.
  In particular, for $k \in \{0, 1, \lambda, \infty \}$, we have that the image of the point $\mathcal{L}_k = [\mE_*(L_k)]$ by $\varepsilon$ is the point $(k, k) \in \P^1_{\tilde{z}} \times \P^1_{\tilde{w}}$.

  Now we are in position to define our coordinate system $\varepsilon$ on $\Buninf^=(C,T)$. Let $\mP$ be a point in $\Buninf^=(C,T)$. By theorem \ref{th:strict}, $\mP$ represents a unique bundle of the form $(E_1, \bm)$. By Proposition \ref{prop:2sec}, the parabolic direction $m_1$ (resp. $m_2$) corresponds to a pair $\{p, i_{t_1}(p)\}$ (resp. $\{q, i_{t_2}(q)\}$) for $p$, $q \in C$. 
  The \emph{coordinate map} in $\Buninf^=(C,T)$ is the mapping
\begin{equation*}
  \varepsilon: \Buninf^=(C,T) \to \P^1_{\tilde{z}} \times \P^1_{\tilde{w}}
\end{equation*}
defined by $\varepsilon(E_1,\bm) = (\varepsilon_1(p), \varepsilon_2(q))$.

\begin{proposition}
  The coordinate map $\varepsilon: \Buninf^=(C,T) \to \P^1_{\tilde{z}} \times \P^1_{\tilde{w}}$ is 1:1.
\end{proposition}
\begin{proof}
  Each parabolic configuration $\bm = (\varepsilon_1(p), \varepsilon_2(q))$ defines a single parabolic bundle $(E_1, \bm)$, and thus a single element $[E_1, \bm]$ of $\Buninf^=(C,T)$.
\end{proof}
\begin{proof}[Proof of Proposition \ref{prop:1:1}]
    Let $R = \OC(r - \wi)$ be a divisor such that $2R = \wi - t_1 = t_2 - \wi$. Then, the following diagram is commutative:
\begin{equation*}
\begin{tikzcd}[column sep = small]
 & \Buninf^=(C,T) \arrow{rd}{\elem^+_{t_2} \otimes \O(R + \wi)}     &  \\ 
\BunO^=(C,T) \arrow{rr}{\phi_T} \arrow{ru}{\elem^+_{t_1} \otimes \O(R)} \arrow{d}{\Tu \circ \Forget} & & \BunO^=(C,T) \arrow{d}{\Tu \circ \Forget}\\
\P^1_z &          & \P^1_w 
\end{tikzcd}
\end{equation*}
  Elements of $\Buninf^=(C,T)$ are of the form $[E_1, \bm = (m_1,m_2)]$. Fix the second parabolic direction $m_2$. For each $m_1 = \varepsilon_1(p)$, we define a point $z(m_1) \in \P^1_z$ as follows: 
\begin{align*}
  z(m_1) := \left (\Tu \circ \Forget \circ \elem^-_{t_1} \otimes \O(-R) \right) [E_1, \bm] \\
= \Tu \left( \O(p - r) \oplus \O(\i_{t_1}(p) -r) \right). 
\end{align*}
  We see that the map $m_1 \mapsto z(m_1)$ is bijective. We also define $w(m_2)$ fixing this time $m_1$, the map $m_2 \mapsto w(m_2)$ is also bijective. Since the map $\elem^+_{t_1} \otimes \O(R)$ is an isomorphism of moduli spaces, the coordinate map 
\begin{align*}
  \varepsilon: \BunO^=(C,T) \to \Buninf^=(C,T)
\end{align*}
is 1:1.

\end{proof}

\paragraph{The relation between even and odd coordinates}

  Here we describe the coordinate change between even and odd degree moduli spaces. Remark that we have seen in the proof of Proposition \ref{prop:1:1} that this coordinate change is not canonical, since we have to choose a root of $t_1$. Let $R = \O(r - \wi)$ be a divisor such that $2R = \O(t_1 - \wi)$. The elementary map
\begin{align*}
  \BunO^=(C,T) \xrightarrow{\elem_{t_1}^+ \otimes \O(R)} \Buninf^=(C,T)
\end{align*}
yields a coordinate transformation $\theta: \PP \to \P^1_{\tilde{z}} \times \P^1_{\tilde{w}}$ satisfying 
\begin{align*}
  \theta \circ \varepsilon = \varepsilon \circ \left( \elem_{t_1}^+ \otimes \O(R) \right).
\end{align*}
 For $k \in \Delta = \{0,1,\lambda,\infty\}$, define the points $p_k := \i_{w_k}(r)$ and $q_k := \i_{w_{\infty}} (p_k)$.

\begin{proposition}
  The map $\theta$ is given by $\theta = \theta_1 \times \theta_2$, where $\theta_1$ and $\theta_2$ are the unique automorphisms of $\P^1$ satisfying
\begin{align*}
  \theta_1(\pi(p_k)) = \theta_2(\pi(q_k)) = k  
\end{align*}
for every $k \in \Delta$.
\end{proposition}

\begin{proof}
  Let $k,l \in \Delta$ and $[L \oplus L^{-1}, \bm] = \varepsilon^{-1} (\pi(p_k),\pi(q_l))$. In the generic situation, there exists $M$ in $\JacC$ such that $\bm$ is defined by the intersection of unique subbundles $M(-\wi)$ and $M^{-1}(-\wi)$ of $L \oplus L^{-1}$. By the definition of $\varepsilon$, we have that $L = \O(p_k - \wi)$ and $M = \O(q_l - \wi)$. 

  The image of $[L \oplus L^{-1}, \bm]$ by $\elem_{t_1}^+ \otimes \O(R)$ is a parabolic bundle $[E_1, \bn = (n_1, n_2)]$, where $n_1$ is the intersection of $L(R)$ and $L^{-1}(R)$ and $n_2$ is the intersection of $M(-R)$ and $M^{-1}(-R)$. By our choice of $p_k$ and $q_l$, we have also that $L(R) = \O(w_k - \wi)$ and $M(-R) = \O(w_l - \wi)$. By the definition of $\varepsilon$, we have that $\varepsilon[E_1, \bn] = (k,l)$.

  We can repeat the arguments in this proof fixing $L$ and varying $M$, or viceversa. In particular, $\theta = \theta_1 \times \theta_2$.
\end{proof}

\section{The group of automorphisms and the Torelli result} \label{sec:aut}

  We are now interested in simple parabolic bundles over $(C,T)$. The following Proposition:
\begin{proposition}
  Let $(E,\bp)$ be a rank 2 parabolic bundle over $(C,T)$. Then $(E,\bp)$ is simple if, and only if, it is $\bmu$-stable for some $\bmu$.
\end{proposition}
\begin{proof}
	Since simplicity is preserved by elementary transformations and twists, we can assume $\det E = \O(\wi)$. Bundles in $\Buninfm(C,T) \bigcup \BuninfM(C,T)$ are listed in Proposition \ref{prop:outsideodd}: $E_1$ is (vectorially) simple by Theorem \ref{th:E1}, thus $(E_1,\bp)$ is simple for every $\bp$. The remaining bundles are of the form $(L \oplus L^{-1}(\wi),\bp)$ with exactly one parabolic direction on $L$. The automorphism group of $(L \oplus L^{-1}(\wi))$ is of projective dimension 1 acting on the complementary of $L$ and $L^{-1}$. Hence, only the identity fixes the parabolic directions $\bp$ and $(L \oplus L^{-1}(\wi), \bp)$ is simple.

  Conversely, $(E,\bp)$ simple implies $(E,\bp)$ is $\bmu$-semistable for some $\bmu$. Parabolic $\bmu$-semistable bundles which are non-$\bmu$-stable for any $\bmu$ are of the form $\mE_=(L)$ by Theorem \ref{th:strict} and Proposition \ref{prop:outsideodd}. These are not simple. 
\end{proof}

  Simple bundles are then parametrized by the non-separated scheme
\begin{align*}
\BunO (C,T)= {X_< \coprod X_>} \Big/ \sim
\end{align*}
constructed by patching the two charts $X_ 1 = \BunOm(C,T)$ and $X_> = \BunOM(C,T)$, where we identify identical parabolic bundles along $X_< \setminus \Gamma$ and $X_> \setminus \Gamma$. 
  
 In this Section, we will study some of the automorphisms of this moduli space and in Section \ref{sec:delpezzogeom} we will prove that there are no more. Automorphisms of $\BunO(C,T)$ consist of pairs of maps $(\psi_1, \psi_2)$ that coincide on the gluing locus and that satisfy one of the following two conditions:
\begin{itemize}
  \item Each $\psi_k$ is an automorphism of $X_{*}$ leaving invariant $\Gamma$, or
  \item The maps $\psi_1:X_< \to X_>$ and $\psi_2:X_> \to X_<$ are isomorphisms leaving invariant $\Gamma$.
\end{itemize}
  By analytic continuation, it is sufficient to explicit only one of the $\psi_k$ to completely define the corresponding automorphism.

\paragraph{Twist automorphisms}
  The twist by a torsion line bundle $L_k$ induces an automorphism between the charts
\begin{equation*}
  X_< \xrightarrow{\otimes L_k} X_< \ , \qquad X_> \xrightarrow{\otimes L_k} X_>.
\end{equation*}
  It is clear by construction that these isomorphisms coincide on the complement $X_k \setminus \Gamma$ of the strictly semistable loci, thus they extend to a global isomorphism 
\begin{equation*}
  \BunO(C,T) \xrightarrow{\otimes L_k} \BunO(C,T).
\end{equation*}
  Since the $\bmu$-stability index is left unchanged under $\otimes L_k$, this automorphism preserves $\Gamma$. We will call this mappings \emph{twist automorphisms}.
  Remark that $\otimes L_{\infty}$ is the identity map and $\otimes L_0 \circ \otimes L_1 = \otimes L_{\lambda}$. 

  Our aim now is to compute these twists in terms of the coordinate system $\varepsilon$ defined in Section \ref{sub:coordinatepair}. 

\begin{proposition}
  Let $k \in \{0,1,\lambda, \infty \}$. Then, in the coordinate system given by $\varepsilon$, the mapping $\otimes L_k$ is expressed as follows:
\begin{align*}
  \PP &\xrightarrow{\otimes L_k} \PP \\
  (z,w) &\mapsto (\beta_k(z), \beta_k(w))
\end{align*}
with $\beta_k$ the map defined in Proposition \ref{prop:beta}. 
\end{proposition}
\begin{proof}
  Let $(E,\l)$ be a parabolic bundle in $X_<$. We will now compute the images  $$E = \Forget (E,\l) \quad \text{and} \quad F = \Forget \circ \phi_T (E,\l).$$

  The twist by $L_k$ in $\BunOm(C,T)$ corresponds to tensoring each of these bundles by $L_k$. Hence, the first (resp. the second) argument of the mapping $\otimes L_k$ depends only on the first (resp. the second) coordinate. Finally, these maps are the twists appearing in Proposition \ref{prop:beta}, i. e. those making the diagram 
\begin{equation*}
\begin{CD}
                              \BunO(C)    @>\otimes L_k>>      \Bun_{\O}(C)\\
                              @V{\Tu}V\cong V         @V\cong V{\Tu}V\\
                              \P ^1  @>\beta_k>>       \P^1\\
\end{CD}
\end{equation*}
commute.
\end{proof}

\paragraph{The elementary automorphism}

  Consider the map $\phi_T$ introduced in Section \ref{sub:coordinatepair}. Since every bundle in $X_<$ also appears in $\BunO^=(C,T)$, $\phi_T$ is well defined on $X_<$. Recall that $\elem^+_T$ changes the weights $\mu_i$ into $1 - \mu_i$. It follows that $\phi_T$ permutes $X_<$ and $X_>$. We will hence consider the map $\phi_T: X_< \to X_>$.
  The description of this map in our setting is the following:
\begin{proposition}
  In the coordinate system defined by $\varepsilon$, the mapping $\phi_T$ is just the mapping that exchanges coordinates between the two factors, i. e. it is the map
\begin{equation*}
  \PP \xrightarrow{\phi_T} \PP 
\end{equation*}
defined by $\phi_T (z,w) = (w, z)$.
\end{proposition}

\begin{proof}   
  Let $(z,w) = ([E = L \oplus L^{-1}], [M \oplus M^{-1}])$ be a point in $\PP$. In the generic situation, $L \not = L^{-1}$ and there are unique subbundle inclusions $M(-\wi) \subset E$ and $M^{-1}(-\wi) \subset E$ defining a parabolic configuration $(E, \bm)$ over $(X,T)$ with both parabolic directions outside $L$ and $L^{-1}$.

  Then, by the properties of elementary transformations, we have $$\phi_T([L \oplus L^{-1}], [M \oplus M^{-1}]) = ([M \oplus M^{-1}], [L \oplus L^{-1}])$$ as stated. 
\end{proof}

  In particular, we have that the curve $\Gamma \subset \PP$ is invariant under the transformation $(z,w) \mapsto (w,z)$.

\paragraph{The odd degree case}
  The former automorphisms can also be defined on the odd degree case, namely for the moduli space $\Buninf(C,T)$. We will show that they have the same coordinate expression than in the even degree case.

\begin{proposition}
  The twist automorphism
  $\Buninf(C,T) \xrightarrow{\otimes L_k} \Buninf(C,T)$
is expressed as follows in the coordinate system $\varepsilon$:
 \begin{align*}
  \P^1_{\tilde{z}} \times \P^1_{\tilde{w}} &\xrightarrow{\otimes L_k} \P^1_{\tilde{z}} \times \P^1_{\tilde{w}} \\
  (\tilde{z},\tilde{w}) &\mapsto (\beta_k(\tilde{z}), \beta_k(\tilde{w}))
\end{align*}
with $\beta_k$ the map defined in Proposition \ref{prop:beta}.
\end{proposition}
\begin{proof}
  It is enough to prove the assertion on $X_<$. Let $[E_1, \bm] \in X_<$ be an arbitrary point. By Proposition \ref{prop:2sec}, the direction $m_j$ is defined by two line subbundles $\O(p_j - \wi)$ and $\O(\i_{t_j}(p_j) - \wi)$, for $j \in \{1,2\}$. 
  Twisting by $L_k$ yields the involution of Proposition \ref{prop:beta} on each of these subbundles.
\end{proof}
  The map $\phi_T$ is defined in the same way as in the even degree case: it is the elementary transformation $\elem_T^+$ followed by the twist by $\O(-\wi)$.
\begin{proposition}
  In the coordinate system  $\varepsilon$, the mapping $\phi_T$ is  the map
\begin{equation*}
  \P^1_{\tilde{z}} \times \P^1_{\tilde{w}} \xrightarrow{\phi_T} \P^1_{\tilde{z}} \times \P^1_{\tilde{w}} 
\end{equation*}
defined by $\phi_T (\tilde{z},\tilde{w}) = (\tilde{w}, \tilde{z})$.
\end{proposition}

\subsection{The Torelli result} \label{sec:Torelli}
  By Theorem \ref{th:strict}, the strictly $\bmu$-semistable locus $\Gamma$ in the moduli space $\BunO^=(C,T)$ is an embedding of $\JacC$, which is itself isomorphic to $C$. Hence, the curve $C$ is naturally embedded in the moduli space. Here we show that this embedding also contains the information about the divisor $T$.
  
  Let $\varepsilon$ be the coordinate system defined in Section \ref{sub:coordinateimpair}. Let $(z_0,w_0) \in \PP \cong \BunO^=(C,T)$ be a point in the moduli space. The vertical and horizontal lines are defined as the subsets
\begin{align*}
  V_{z_0} = \{(z_0,w) \ | \ w \in \P^1_w \} \ , \
  H_{w_0} = \{(z,w_0) \ | \ z \in \P^1_z \}
\end{align*}
\begin{lemma} 
  Let $\Delta = \{0,1,\lambda,\infty \}$. Then, the vertical line $V_{z_0}$ is tangent to $\Gamma$ if, and only if, $z_0 \in \Delta$. Similarly, the horizontal line $H_{w_0}$ is tangent to $\Gamma$ if, and only if, $w_0 \in \Delta$.
\end{lemma}
\begin{proof}
  For $z_0 = [L \oplus L^{-1}] \in \P^1_z$ fixed, there are generically two points in $\Gamma$ with first coordinate $z_0$, namely $[\mF_*(L)]$ and $[\mF_*(L^{-1})]$. Therefore, $V_{z_0}$ is tangent if, and only if $L$ is a 2-torsion bundle, which correspond to $z_0 \in \Delta$ according to Proposition \ref{prop:Tu}.

  Since $\Gamma$ is invariant under the transformation $(z,w) \mapsto (w,z)$, $H_{w_0}$ is a tangent horizontal line if, and only if, $V_{w_0}$ is a tangent vertical line.  
\end{proof}
  The following Theorem is the explicit version of Theorem \ref{thx:Torelli}:
\begin{proposition} \label{prop:tang}
  Let $k \in \Delta$ and $V_{k}$ be a vertical tangent to $\Gamma$. Let $(k, w_0) \in \PP$ be the tangency point. Then, there exists an element $g \in \text{PGL}(2)$ such that
\begin{align*}
  g(\{0,1,\lambda,\infty\})=\{0,1,\lambda,\infty\} \text{ and } g(w_0) = t.
\end{align*}
\end{proposition}

\begin{proof}
  Assume $k = 0$ and let $\mP = (0,w_0)$ be the tangency point with $\Gamma$. By definition of $\varepsilon$, $\mP$ represents parabolic bundles with underlying bundle $\O \oplus \O$ or $E_0$. Since $\mP$ is in $\Gamma$, we have that $\mP = [\mF_*(\O)]$. 
  
  Consider the bundle $\mF_<(\O) = (E_0, \bp)$, where $p_1 \subset \O \subset E_0$ and $p_2 \not \subset \O$. Let $q \in C$ be the unique point such that the subbundle $L = \O(-q)$ passes through both $p_1$ and $p_2$. 

  Let us apply the mapping $\phi_T$ to $\mF_<(\O)$. After the first elementary transformation $\elem^+_{t_1}$, we get the bundle $\mE = (\O(t_1) \oplus L(t_1), \bp')$, where $p_1'$ is outside both factors and $p_2'$ lies on $L(t_1)$. Therefore, the underlying bundle of $\elem^+_{t_2}(\mE)$ is $\O(t_1) \oplus L(t_1 + t_2)$. Twisting by $\O(-\wi)$ gives the bundle $F = \O(t_1 - \wi) \oplus L(\wi)$. Since $\text{det} F = \O$, it follows that $q = t_1$ and  $w_0 = \Tu[F] = t$, thus we choose $g = \text{id}$.
  
  For a $k \in \Delta \setminus \{0\}$, the above discussion holds if we multiply every bundle by $L_k$. The final underlying bundle is $F_k = F \otimes L_k$. By Proposition \ref{prop:beta}, $w_0 = \Tu[F_k] = \beta_k(t)$. Hence, we can take $g = \beta_k^{-1} = \beta_k \in \text{PGL}(2)$.
\end{proof}

\begin{proof}[Proof of Theorem \ref{thx:Torelli}]
  Let $\mathbf{G} = \text{PGL} (2) \times \text{PGL} (2)$. We have to show that there are one-to-one correspondences between the sets
\begin{align*}
\left\{\parbox[p]{6em}{\begin{center} $(2,2)$-curves   $\Gamma \subset \PPsin$\end{center}}\right\} \Big/ G \
\xleftrightarrow{\ 1:1 \ }
\left\{\parbox[p]{8em}{\begin{center} $2$-punctured elliptic curves  $(C, T)$ \end{center}}\right\} \Big/\sim & \
\xleftrightarrow{\ 1:1\ }
\left\{\parbox[p]{8em}{\begin{center} $4+1$-punctured rational curves  $(\P^1, \uD + t)$\end{center}}\right\} \Big/\sim.
\end{align*}
  For the first correspondance, consider the elliptic curve $\Gamma \subset \P^1 \times \P^1$ and the elliptic cover given by the first projection  $\pi: \P^1 \times \P^1 \to \P^1$. The divisor $T$ is given by the preimage of the point $t$ defined in Proposition \ref{prop:tang}. This divisor is invariant by $\mathbf{G}$ again by Proposition \ref{prop:tang}. The inverse map is given by our construction of the moduli space $\BunO(C,T)$.
  
  The second correspondance is given by the elliptic cover $\pi: C \to \P^1$ such that $\pi(t_1) = \pi(t_2) = t$, and its ramification divisor $\uD$.
\end{proof}

\section{A map between moduli spaces} \label{sec:map}  
  Let $\uW = 0+1+\lambda+\infty$ and $\uD = \uW + t$ be reduced divisors on $\P ^1$. Let $C$ be the curve defined by the equation $y^2 = x (x-1) (x-\lambda)$. Consider the divisors $W = w_0 + w_1 + w_\lambda + \wi$ and $T = t_1 + t_2$ on $C$ respectively supported by the Weierstrass points and by the two preimages of $t$ under the hyperelliptic cover $\pi$. Let $D = W + T$.

  We study in this Section a map $\Phi$ between the moduli spaces $\Bundemi$ and $\BunOCT$, with fixed weights $\ubmu$ and $\bmu$. The first space is the moduli space of $\ubmu$-semistable rank 2 parabolic vector bundles of degree 0 over $(\P^1, \uD)$. This moduli space is described and constructed in ~\cite{loraysaito}. 

  More precisely, the authors consider the full coarse moduli space  $\Bun_{-1}(\P^1, \uD)$ of degree $-1$ parabolic bundles ($P_{-1}(\boldsymbol{t})$ in their notation). They construct this space by patching projective charts consisting on moduli spaces $\Bun_{-1}^{\bnu}(\P^1, \uD)$ of $\bnu$-semistable bundles.
  
  Let us describe two of these charts. The chart $V$ corresponds to the moduli space $\Bun_{-1}^{\bnu}(\P^1, \uD)$, for \guillemotleft democratic\guillemotright weights $\nu_i = \nu$, with $\frac{1}{5} < \nu < \frac{1}{3}$, and is isomorphic to $\P^2$. It consists of those indecomposable parabolic bundles of the form $E = (\OP \oplus \OP(-1), \bn)$ with no $n_i$ lying in $\OP$ and not all $n_i$ lying in $\OP(-1)$ (see Proposition 3.7 of ~\cite{loraysaito}). There are 16 special geometric objects in $V$, namely: 
\begin{itemize}
  \item Five points $D_i$ for $i \in \{0,1,\lambda,\infty,t\}$. 
  \item Ten lines $\Pi_{ij}$ joining $D_i$ and $D_j$.
  \item The conic passing through all $D_i$.
\end{itemize}
  Let $\tOmega$ be the set of these objects. 
  The chart $\S$ corresponds to the moduli space with democratic weights $\bnu$ such that $\frac{1}{3} < \nu < \frac{3}{5}$. As a projective surface, it is isomorphic to the blow-up of the five points $D_i$ in $\P^2$. This is by definition the del Pezzo surface of degree 4 (see ~\cite{dolgachev}). In particular, the exceptional divisors $\Pi_i$ of $D_i$ and the total transforms of $\Pi_{ij}$ and $\Pi$ constitute 16 ($-1$)-curves in $\S$. We will keep the notation $\Pi_{ij}$ and $\Pi$ for the total transforms when there is no risk of confusion. There are exactly five ($-1$)-curves intersecting $\Pi$, namely the $\Pi_i$'s, and the chart $V$ is obtained as the blow-down of these 5 curves. Similarly, there are 5 ($-1$)-curves intersecting $\Pi_i$, namely $\Pi$ and $\Pi_{ij}$ for $j \not = i$. The other charts are isomorphic to projective surfaces obtained by blow-downs of some of these $(-1)$-curves. 

  The open set
\begin{align*}
  U_{-1} = V \setminus \left\{ D_i , \Pi_{ij}, \Pi \right\}
\end{align*}
of generic parabolic bundles is common to every chart. The final non-separated patching is made via the blow-down birational maps between the charts  (see Theorem 1.3 in ~\cite{loraysaito}). 
  
\subsection{Defining the map $\Phi$} \label{sec:mapPhi}
  Let us start by fixing weights $\ubmu = (\frac{1}{2},\frac{1}{2},\frac{1}{2},\frac{1}{2},\mu)$, where the free weight $\mu$ corresponds to the point $t$. Consider the associated moduli space $\Bundemi$. The isomorphism $$\elem_0^+:\Bun_{-1}^{\ubmu}(\P^1, \uD) \to \Bundemi$$ is well-defined by the properties of elementary transformations. Let $U_0$ be the image of $U_{-1}$ by this map. We have the following result:  

\begin{proposition}
  The moduli space $\Bundemi$ is isomorphic to $\S$ for all $\mu \in [0,1]$.  
 Every bundle in $U_0$ has trivial underlying bundle.
\end{proposition}

\begin{proof}
  The map $\elem_0^+$ is an isomorphism of moduli spaces preserving $\ubmu$-stability. A straightforward calculation shows that all bundles in the families $\Pi$, $\Pi_i$, and $\Pi_{i,j} \subset \Bun_{-1}^{\ubmu}(\P^1, \uD)$ are $\ubmu$-semistable (stable when $0 < \mu < 1$). Hence, $\Bundemi  \cong \S$.
  
  For the second assertion, remark that every parabolic bundle $\mE$ in $U_{-1}$ can be written as $(\OP \oplus \OP(-1), \bl)$ with $l_0 \subset \OP(-1)$. By properties of elementary transformations, $\elem_0^+(\mE)$ has trivial underlying bundle.
\end{proof}

  Let us now define the mapping $\Phi$. Let $\mE$ be a parabolic bundle in $\Bundemi$. Consider the pullback bundle $\pi^*\mE$ of $\mE$: it is a parabolic bundle over $(C, D)$. The bundle $\pi^*\mE$ is $\ubmu'$-semistable, where $\ubmu' = (1,1,1,1,\mu,\mu)$.

  Consider the following composition of maps:
\begin{align*}
  \BunO^{\ubmu'}(C, D) \xrightarrow{\elem_{W}^+} \Bun^{\ubmu''}_{4\wi}(C, D) \xrightarrow{\text{Forget}_W} \Bun^{\bmu}_{4\wi}(C, T) \xrightarrow{\otimes M} \BunOCT.
\end{align*}
  The weights here are $\ubmu'' = (0,0,0,0,\mu,\mu)$ and $\bmu = (\mu,\mu)$.
  The first map is the positive elementary transformation over $W$. 
  The second map forgets parabolic directions over $W$ and keeps those over $T$. Because of the nullity of the weights over $W$, this map preserves the stability notion.
  The last map is the twist automorphism, with $M = \OC(-2 \wi)$.

  Let $\phi_W$ be the composition of these four maps.
  The map $\Phi$ is the composition of $\pi^*$ and $\phi_W$:

\begin{equation*}
\begin{tikzcd}[column sep = small]
 \Bundemi \arrow{rr}{\Phi} \arrow{rd}{\pi^*}	&	& \BunOCT \\
 	        & \BunO^{\ubmu'}(C, D) \arrow{ru}{\phi_W}    &  
\end{tikzcd}
\end{equation*}

  Notice that the role of $t \in \P^1$ in the definition of this map is different from the roles of the other points in $\uD$. We have defined a morphism $\Phi:\Bundemi \to \BunOCT$ between our moduli spaces.

\subsection{Computing the map $\Phi$}
  First, we will compute the map $\Phi|_{U_0}$. For $c, l \in \P^1$, consider the set 
\begin{align*}
  U_C := \left\{ (\OC \oplus \OC, \underline{\bm}) \ | \  c, l  \in \P^1 \right\} \subset \BunO^{\ubmu'}(C, D). 
\end{align*}
of parabolic bundles on $(C, D)$, where  $\underline{\bm} := (0, 1, c, \infty, l, l)$. 
  Since $\pi(t_1) = \pi(t_2)$, the parabolic directions over $t_1$ and $t_2$ of $\pi^*\mE$ are the same. Thus, $\pi^*$ is a birational map between $U_0$ and $U_C$. 

  We will use the coordinate chart $U_C \cong \P^1_c \times \P^1_l$ and the coordinate system $\varepsilon$ for $\BunOCT$ defined in Section \ref{sub:coordinatepair}.
\paragraph{The map $\phi_W|_{U_C}$}
  Let $(E, \bm)$ be a parabolic bundle over $(C, D)$ and $D_0 \subset D$ a subdivisor. We say that a line subbundle $L$ of $E$ passes through $D_0$ if $L$ passes through every point of $D_0$.
\begin{lemma} \label{lem:parlinbun}
  Let $\mE = (\OC \oplus \OC, \bm)$ be a parabolic bundle over $(C, D)$. 
  Let $S_W^{-2}(\mE)$ be the set of degree $-2$ line subbundles  of $\OC \oplus \OC$ passing through $\bm|_W$. 
  Then, there exists a point $p \in C$ such that
\begin{align*}
  S_W^{-2}(\mE) = \{ L_W, L_W' \} \ \text{with} \ L_W \cong \OC(-p-\wi) \text{, } L_W' \cong \OC(-\i_{\wi}(p)-\wi)
\end{align*}
  Similarly, let $S_{D}^{-3}(\mE)$ be the set of degree $-3$ line subbundles  of $\OC \oplus \OC$ passing through $\bm|_{D}$. 
  Then, there is a point $q \in C$ such that
\begin{align*} 
  S_{D}^{-3}(\mE) = \{ M_D, M_D' \} \ \text{with} \ M_D \cong \OC(-q-2\wi) \text{, } M'_{D} \cong \OC(-\i_{\wi}(q)-2\wi) .
\end{align*}
\end{lemma}
\begin{proof}
  Let us prove the assertion for $S_W^{-2}(\mE)$. Consider a line bundle $L = \OC(-\tilde{p}-\wi)$, for $\tilde{p}$ in $C$. The vector space $\Hom(L, \OC \oplus \OC)$ is isomorphic to $\Hom(\OC, L^{-1} \oplus L^{-1})$. The projective dimension of this space is 3 by Riemann-Roch. This means that there is an inclusion $L \xhookrightarrow{} \OC \oplus \OC$ such that $L$ passes through three parabolic directions, say $m_0$, $m_1$ and $m_{\lambda}$. By moving $\tilde{p}$ in $C$ we move on the fiber of $\wi$. This constitutes a 2:1 covering of $\P^1$, hence there are generically two points $p$ and $p'$ in $C$ such that $S_W^{-2}(\mE) = \{ L, L' = \OC(-p'-\wi) \}$. To see that $p' = \i_{\wi}(p)$, consider the case $p \not = p'$. We have that $\Forget \circ \elem_W^+(\mE) \otimes \OC(-2 \wi) = \OC(\wi-p) \oplus \OC(\wi-p')$ has trivial determinant, implying $p' = \i_{\wi}(p)$.  
  The proof for $S_{D}^{-3}(\mE)$ is similar. 
\end{proof} 

\begin{proposition} \label{prop:Phi}
  The map $\phi_W|_{U_C}$ is given by
\begin{align*}
\P ^1_c \times \P^1_l \xrightarrow{\phi_W|_{U_C}} \PP  \quad , \quad 
 \phi_W|_{U_C} (c , l) =  \left(\phi_W|_{U_C}^1(c,l),\phi_W|_{U_C}^2(c,l)\right) 
\end{align*}
where 
\begin{equation*}
 \phi_W|_{U_C}^1(c,l) = \lambda \frac{c-1}{c-\lambda} \quad \text{and} \quad 
 \phi_W|_{U_C}^2(c,l) = \frac{\lambda l (\lambda (l - 1) + t (1-c))}{\lambda\left(t(l - c) + l(c - 1)\right) + ct(1-l)}.
\end{equation*}
\end{proposition}

\begin{proof}
  Consider the trivial rank 2 vector bundle $\O_C \oplus \O_C$ and its projectivization $C \times \P^1$. Let $\bm = (0,1,c,l,l,\infty)$ be a parabolic configuration on $\O_C \oplus \O_C$.
  Let us write $(E, \bn) := \phi_W|_{U_C}(\OC \oplus \OC, \bm)$. We want to compute the $\varepsilon$-coordinates of $(E, \bn)$ in terms of $(c,l)$.  

  There exists $p \in C$ such that $[E] = [L\oplus L^{-1}]$ in $\BunO(C)$, with $L = \O(p - \wi)$. 
  By the properties of elementary transformations, the preimages of the line bundles $L$ and $L^{-1}$ via $\phi_W|_{U_C}$ are the following subbundles of $\OC \oplus \OC$:
\begin{align*}
  L_W = \OC(- p - \wi) \quad , \quad L'_W = \OC( - \i_{\wi}(p) - \wi)
\end{align*}
By construction, $L_W$ and $L'_W$ pass through the parabolic directions over $W$. Moreover, these are the unique subbundles of degree $-2$ passing through these parabolic directions by Lemma \ref{lem:parlinbun}. 

  Consider the $(+4)$-cross-sections $s_W$ and $s'_W$ of the trivial projective bundle $C \times \P^1$ corresponding to $L_W$ and $L'_W$ respectively. The section $s_W$ (resp. $s'_W$) intersects the constant section $y = \infty$ in two points with base coordinates $\wi$ and $p$ (resp. $\wi$ and $\i_{\wi}(p)$) in $C$. We claim that $\pi(p) =  \lambda \frac{c-1}{c-\lambda}$. This will imply
\begin{align*}
  \phi_W|_{U_C}^1(c,l) = T \circ \Forget \circ \phi_W|_{U_C}(\OC \oplus \OC, \bm) = \lambda \frac{c-1}{c-\lambda}
\end{align*}
as desired. 
  In order to prove the claim, consider the pullback of the section $s_W$ by $\pi$. It is a curve of bidegree $(2,2)$ in $\P(\OP \oplus \OP) = \PP$ having vertical tangencies over the ramification points $0, 1, \lambda, \infty \in \P^1_z$ and intersecting the horizontal $w = \infty$ in two points with base coordinates $z_0 = \infty$ and $z_1 = \pi(p)$ respectively (see Figure \ref{fig:Pzw}). This curve is defined by a bidegree (2,2) polynomial $P(z,w)$ satisfying the former conditions, which translate to 8 polynomial equations on the coefficients of $P$. Solving this polynomial system yields $z_1 = \lambda \frac{c-1}{c-\lambda}$. 
  
\begin{figure}
    \centering
    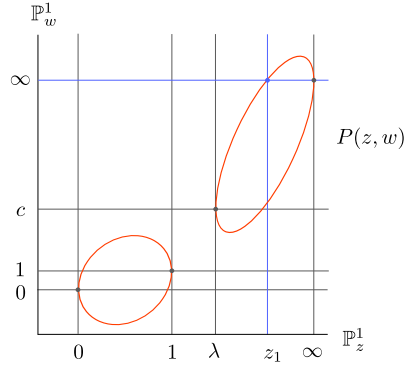
\caption{The curve of bidegree $(2,2)$ defined by $P(z,w) = 0$. Here, $z_1 = \lambda \frac{c-1}{c-\lambda}$.}
\label{fig:Pzw}
\end{figure}
  
  For the second coordinate, we have that 
\begin{align*}
  \phi_W|_{U_C}^2(c,l) = T \circ \Forget \circ \phi_T \circ \phi_W|_{U_C}(\OC \oplus \OC, \bp) = T [M \oplus M^{-1}],
\end{align*}
with $M = \O(q - \wi)$ for $q \in C$. The preimages of $M$ and $M^{-1}$ by $\phi_T \circ \phi_W|_{U_C}$ are 
\begin{align*}
  M_{D} = \OC(- \i_{\wi}(p) - 2 \wi) \quad , \quad M'_{D} = \OC(- p - 2 \wi)
\end{align*}

  These correspond to two $(+3)$-cross-sections $s_{D}$ and $s_{D}'$ of the projective bundle passing through every parabolic point over $W + T$. The section $s_{D}$ (resp. $s_{D}'$) intersects the constant section $y = l$ in three points having base coordinates $t_1$, $t_2$ and $\i_{\wi}(p)$ (resp.  $t_1$, $t_2$ and $p$). The pullback of these cross-sections by $\pi$ is a curve of bidegree $(3,2)$ on $\PP$ having vertical tangencies over the ramification points and crossing in a node over $t$. Applying these conditions to a polynomial of degree (3,2) on $z$, $w$ yields the desired formula for $\phi_W|_{U_C}^2(c,l)$.
\end{proof}

\paragraph{The map $\Phi$}
  Let $\Pb$ be the projective line with homogeneous coordinates $\bb = (b_0 : b_1 : b_2)$. The birational coordinate change $\P^1_c \times \P^1_l \to \Pb \supset U_0$ is explicited in Section 6 of ~\cite{loraysaito}. 
  Composing with the coordinates of $\phi_W|_{U_C}$ of Proposition \ref{prop:Phi} of ~\cite{loraysaito}, we find that the mapping
\begin{align*}
  \Phi|_{U_0}:  U_0 \to \PP  \cong \BunOCT   
\end{align*} 
is given in $U_0$ by the expression $$\Phi|_{U_0} (b_0:b_1:b_2) = \left( \frac{b_1 t - b_2}{b_0 t - b_1} , -b_1 \frac{b_0 \lambda - b_1 \lambda - b_1 + b_2}{b_1^2 - b_0 b_2} \right).$$
  This map extends to a 2:1 rational map $\tPhi:\Pb \to \PP$ by analytic continuation.

\subsection{The special locus of $\tPhi$ and the geometric configuration on $\Pb$}
  In this Section we will relate the set $\tOmega$ of 16 geometric objects in $\Pb$ with the undeterminacy and the ramification locus of $\tPhi$. 

\paragraph{Special locus of $\tPhi$} The undeterminacy locus of $\tPhi$ consists of the five special points $D_i \in \Pb$.  Their projective coordinates are
\begin{equation*}
D_0 = (1:0:0), \quad D_1 = (1:1:1), \quad D_\lambda = (1:\lambda:\lambda^2), \quad D_t = (1:t:t^2), \quad D_\infty=(0:0:1).
\end{equation*}
  The conic $\Pi$ passing by all the $D_i$ is given by the equation $ \Pi : b_1^2 - b_0 b_1$.
  A calculation shows that the map $\tPhi$ ramifies over the cubic $\Sigma \subset \P ^2 _{\bb}$ defined by the equation
\begin{align*}
  \Sigma:-b_{0}^{2} b_{1} {\lambda} t^{2} + b_{0}^{2} b_{2} {\lambda} t + {\left({\lambda} t^{2} + {\lambda} t + t^{2}\right)} b_{0} b_{1}^{2} - {\left(t^{2} + {\lambda}\right)} b_{1}^{3}  \\ 
- 2 \, {\left({\lambda} t + t\right)} b_{0} b_{1} b_{2} + b_{1}^{2} b_{2} {\left({\lambda} + t + 1\right)} + b_{0} b_{2}^{2} t - b_{1} b_{2}^{2} = 0.
\end{align*}
  The cubic $\Sigma$ passes through the 5 points $D_i$ and is tangent to $\Pi$ in $D_t$. This cubic is precisely the preimage of $\Gamma \subset \PP$. The curve $\Gamma$ has equation
\begin{align*}
\Gamma: t^{2} z^{2} - 2 \, t z^{2} w + t^{2} w^{2} - 2 \, t z w^{2} + z^{2} w^{2} - 2 \, {\lambda} t z - 2 \, {\lambda} t w + 2 \, {\left(2 \, {\left({\lambda} + 1\right)} t - t^{2} - {\lambda}\right)} z w + {\lambda}^{2}
\end{align*}

\paragraph{The action of $\tPhi$ on the 16 objects}
 
  There are four vertical lines $V_i$ and four horizontal lines tangent to $\Gamma$ (see Section \ref{sec:Torelli}): 
$$
  V_i = \{(z,w) \in \PP \ | \ z = i \},   \quad   H_i = \{(z,w) \in \PP \ | \ w = i \} 
$$
  with $i \in \{0, 1, \lambda, \infty \}$. 
  The map $\tPhi$ sends each curve of $\tOmega$ to one of these tangents lines to $\Gamma$ in the following way:
\begin{itemize}
  \item For $i \not = t$, the line $\Pi_{it}$ is sent to $V_i$.
  \item The conic $\Pi$ is sent to $H_{\infty}$.
  \item For $\{i,j,k\} = \{0,1,\lambda\}$, the lines $\Pi_{i\infty}$ and $\Pi_{jk}$ are sent to $H_i$.  
\end{itemize}

\paragraph{The desingularisation map $\Phi$}
  The map $\tPhi$ is birational with base points $D_i$. Hence, there exists a morphism $\Phi$ such that the following diagram commutes:
\begin{equation*}
\begin{tikzcd}[column sep = small]
 \S \cong \Bundemi \arrow{r}{\Phi} \arrow{d}{\text{blow-up}} & \BunOCT \cong \PP \\
 \Pb \arrow{ru}{\tPhi}    
\end{tikzcd}
\end{equation*}
  The vertical map is the blow-up of the 5 points $D_i$ in $\Pb$. The map $\Phi$ defined in Section \ref{sec:mapPhi} is the desingularisation map of $\tPhi$. From the previous discussion, we have the following result:
\begin{theorem}
  The map $\Phi$ is a 2:1 cover of $\PP$ ramified over the curve $\Gamma \subset \PP$.
\end{theorem}
\begin{proof}
  Let us prove that $\S$ is the 2:1 cover of $\PP$ ramified along $\Gamma$.
  Consider the standard Segre embedding $i:\PP \to \P^3_u$ given by
$$i((z_0:z_1),(w_0:w_1)) = (z_0w_0:z_0w_1:z_1w_0:z_1w_1).$$ Its image is the quadric given by the equation $f = u_0u_3 - u_1u_2$. The curve $\Gamma$ is the restriction of a quadric in $\P^3_u$ of equation $g = 0$. The ramified cover of $\P^3_u$ along $g = 0$ is then given by the equation $v^2 = g$ in $\P^4_{u,v}$, and the covering morphism is given by the projection onto $u$.

  The 2:1 covering of $\PP$ is thus given by the quadratic forms $t^2 = g$ and $f=0$ in $\P^4_{u,v}$. Since the intersection is smooth, the covering is a del Pezzo surface of degree 4 (see for example \cite{dolgachev}, Section 8.6).
  Since the map $\Phi$ is also 2:1 with ramification locus $\Gamma$, it is the covering map.
\end{proof}

  Let $\Omega = \{\Pi_i, \Pi_{ij}, \Pi\}$. From the description of $\tPhi$ on $\tOmega$, we get that the elements of $\Omega$  are sent by $\Phi$ to vertical and horizontal tangents to $\Gamma$ as follows (see Figure \ref{fig:doublecover}):
\begin{itemize}
  \item For $i \not = t$, the exceptional divisor $\Pi_i$ and the line $\Pi_{it}$ are sent to $V_i$.
  \item The exceptional divisor $\Pi_t$ and the conic $\Pi$ are sent to $H_{\infty}$.
  \item For $\{i,j,k\} = \{0,1,\lambda\}$, the lines $\Pi_{i\infty}$ and $\Pi_{jk}$ are sent to $H_i$.  
\end{itemize}
  Each $(-1)$-curve is mapped by $\Phi$ bijectively onto its image.
\begin{figure}
    \vspace{1cm}
      \centering
          \def\svgwidth{\columnwidth}
          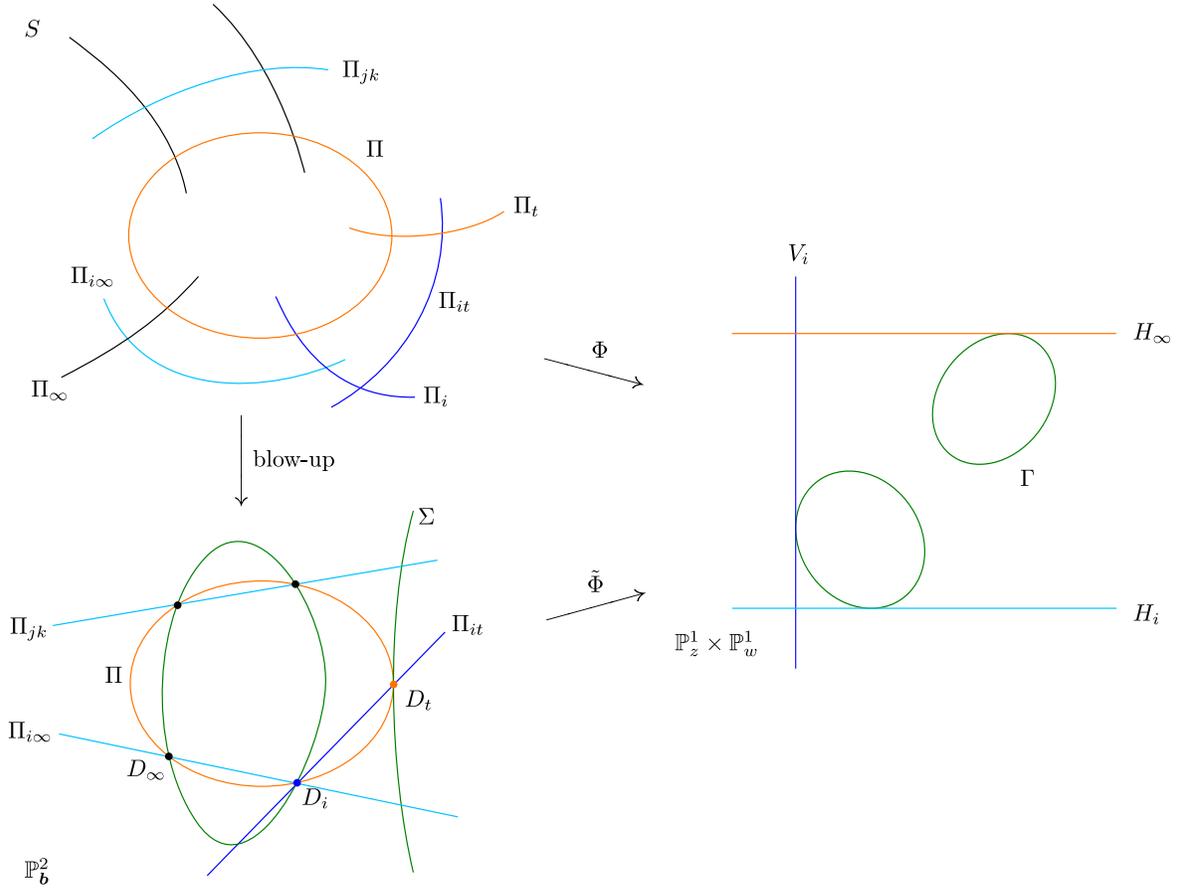
	\caption{The geometry of the map $\Phi$.}\label{fig:doublecover}
\end{figure}

\subsection{The involution $\tau$ of $\Bundemi$}
  Let $(E, \bl) \in \Bundemi$  be a parabolic bundle. Define $$\tau(E, \bl) := \elem^+_{\uW} (E,\bl) \otimes \OP(-2).$$ 
  The bundle $\tau(E,\bl)$ is again an element of $\Bundemi$, therefore the map $\tau$ is an automorphism of this moduli space. 
\begin{proposition}
  The mapping $\tau$ is the involution induced by $\Phi$. More precisely, $\tau$ satisfies $\Phi \circ \tau = \Phi$ and it is not the identity.
\end{proposition}

\begin{proof}
  Locally, $\PE$ is a product $U_x \times \P^1_y$. Consider the parabolic point $p = (0, 0) \in U_x \times \P^1_y$. The pullback map $\pi^*$ is given by $$(x,y) \gets (z,y),$$ where $z$ is a local coordinate in $C$ with $z^2 = x$. An elementary transformation centered in $(z = 0, y = 0)$ corresponds to a coordinate transformation $(z, y')$ with $y' = y/z$.

  Let us now apply an elementary transformation to $(E,p)$. We obtain the coordinate transformation $(x, y'' = y/x)$. The pullback gives $$(x, y'') \gets (z, \tilde{y} = y/z^2).$$ The elementary transformation of $\pi^*(E,p)$ is locally given by $(z, \tilde{y} z = y/z = y')$ (notice that the parabolic point is now at $y'' = \infty$). 

  We have shown that the parabolic bundles $(E, p)$ and $\elem_p(E,p)$ have the same image by pullback followed by elementary transformation on a point. Since the arguments are local, the result holds also for $\phi_T$.
\end{proof}

  The map $\tau$ is a degree 3 birational transformation of $\Pb$ (see Table \ref{table:involutions}). The vanishing locus of its Jacobian determinant is exactly the union of the geometric objects in $\tOmega$.

\paragraph{The de Jonquières automorphism}
  Let $P$ be a point in $\Pb$. The line $\Pi_{P,D_t}$ passing through $P$ and $D_t$ and the conic $\Pi_{P}$ passing through the five points $D_i$ for $i = 0,1,\lambda,\infty$ and $P$ intersect generically in two points, which are $P$ and $\tau(P)$. 

  The involution $\tau$ is given by the intersection of two foliations on $\Pb$: the pencil of lines passing through $D_t$ and the pencil of conics passing through $D_i$, for $i \not = t$. The ramification locus $\Sigma \subset \Pb$ of $\tPhi$ is the tangency locus of these two foliations. The involution $\tau$ leaves invariant $\Sigma$ pointwise. 

  Each point in $\Sigma$ represents a parabolic bundle $(\OC \oplus \OC, \bp)$ fixed by $\tau$. Following Lemma \ref{lem:parlinbun}, this condition means  exactly that $\{L_W, L'_W\} = \{M_W, M'_W\}$. This is equivalent to $L_W$ passing through at least one of the parabolic directions over $t_i$. Thus, $\Sigma$ is exactly the locus of semistable bundles in $\Bundemi$. Since $\bmu$-stability is preserved by elementary transformations, we obtain again that the image $\Phi(\Sigma)$ is the curve $\Gamma$ of Section \ref{sec:moduli}.

  The action of $\tau$ on $\Omega$ is summarized in Table \ref{table:actionomega}.

\section{The del Pezzo geometry and the group of automorphisms} \label{sec:delpezzogeom}
  
  Let $\AutS$ be the group of automorphisms of the del Pezzo surface $\S$. We have the following Theorem:
\begin{theorem}
  The group $\AutS$ is isomorphic to $(\mathbb{Z}/2\mathbb{Z})^4$ and acts transitively and freely on $\Omega$.
\end{theorem}
\begin{proof}
  See, for example, Chapter 8 of ~\cite{dolgachev}.
\end{proof}

  In particular, every automorphism of $S$ is an involution and it is uniquely defined by the image of a $(-1)$ curve in $\Omega$. This set is invariant under $\AutS$. 

\subsection{The group of automorphisms $\Aut(\PP, \Gamma)$}

  In Section \ref{sec:aut} we defined five automorphisms of the moduli space $\BunOCT \cong \PP$: the four twists $\tsigmak: = \otimes L_k$, for $k \in \Delta = \{0,1,\lambda,\infty\}$, and the map $\phi_T$. These maps are involutions preserving the bidegree $(2,2)$ curve $\Gamma$.
  Recall that $\tilde{\sigma_{\lambda}} = \tilde{\sigma_{0}} \circ \tilde{\sigma_1}$.
  
  \begin{table}[htp]
\begin{center}
\begin{tabular}{|c|l|}
\hline
 &
{\small $\tau_0 = (( \lambda t) b_{0} b_{1} + (t^{2} -  t \lambda -  t) b_{0} b_{2} + (- t^{2} -  \lambda) b_{1}^{2} + (t + \lambda + 1) b_{1} b_{2} - b_{2}^{2})(b_0 t - b_1)$ }
\\ $\tau$ &
{\small$\tau_1 = {\left(b_{0} {\lambda} - b_{1} {\lambda} - b_{1} + b_{2}\right)} {\left(b_{0} t - b_{1}\right)} {\left(b_{1} t - b_{2}\right)} t$}
\\ & 
{\small $\tau_2 =\left(b_{0}^{2} + (- t^{2} \lambda -  t^{2} -  t \lambda) b_{0} b_{1} + (t \lambda + t -  \lambda) b_{0} b_{2} + (t^{2} + \lambda) b_{1}^{2} + (- t) b_{1} b_{2}\right) {\left(b_{1} t - b_{2}\right)} t$}
\\ \hline
 & 
$\sigma_{00} =  \left(\lambda b_{0} + (-1 - \lambda + t) b_{1} \right)   (t b_{1} -  b_{2}) $ 
\\ $\sigma_0$ & 
$\sigma_{01} = \lambda  (t b_{0} -  b_{1})  (t b_{1} -  b_{2})  $
\\ & 
$\sigma_{02} = \lambda  (t b_{0} -  b_{1})  \left((- \lambda + t + \lambda t ) b_{1} -  t b_{2} \right) $
\\ \hline

 & 
$\sigma_{10} = (t b_0 - (1+t) b_1 + b_2) \left( ( \lambda - t - 1) b_1 +t b_0 \right)$
\\ $\sigma_1$ & 
$\sigma_{11} = t  (\lambda b_{0} -  b_{1})  \left(t b_{0}+ (-1 -t) b_{1}  + b_{2} \right) $
\\ & 
$\sigma_{12} = t  \left( \lambda^2 t b_0^2 + (\lambda^2+t)b_1^2+(-\lambda^2-\lambda t -\lambda^2t)b_0b_1 + (\lambda-t+\lambda t)b_0 b_2 - \lambda b_1 b_2 \right) $
\\ \hline
 & 
$\sigma_{\lambda 0} = \left(\lambda t b_0 + (1-\lambda-t)b_1 \right)\left(\lambda t b_0 -(\lambda + t)b_1 + b_2\right)$ 
\\ $\sigma_{\lambda}$ &
$\sigma_{\lambda 1} = \lambda t (b_0-b_1) \left(\lambda t b_0 - (\lambda+t)b_1+b_2\right) $
\\ & 
$\sigma_{\lambda 2} = \lambda t \left(
\lambda t b_0^2 + (1 + \lambda t) b_1^2 - (\lambda + t + \lambda t)b_0 b_1 + (\lambda + t - \lambda t)b_0 b_2 - b_1 b_2 
\right)$
\\ \hline
 & 
$\psi_{T0} = (-\lambda - t^2)b_1^2 - b_2^2 + \lambda t b_0 b_1 + (t^2 - t - \lambda t) b_0 b_2 + (\lambda + t + 1) b_1 b_2$ 
\\ $\psi_{T} $ &
$\psi_{T1}  = t \left( \lambda b_0 - (1 + \lambda) b_1 + b_2 \right) (t b_1 - b_2)$
\\ & 
$\psi_{T2}  = t \left( ( t - \lambda + \lambda t) b_1 - t b_2 \right) (\lambda b_0 - (1 + \lambda) b_1 + b_2)$
\\ \hline

\end{tabular}
\end{center}
\caption{The involutions $\tau$, $\sigma_k$ and $\psi_T$ in the projective coordinates $\bb$. Here $\tau =(\tau_0:\tau_1:\tau_2)$,  $\sigma_k = (\sigma_{k0}:\sigma_{k1}:\sigma_{k2})$ and $\psi_T = (\psi_{T0}:\psi_{T1}:\psi_{T2}).$}
\label{table:involutions}
\end{table}
  These involutions act transitively on the set of vertical and horizontal tangents to $\Gamma$: the involution $\phi_T$ exchanges verticals and horizontals, and the twists $\tsigmak$ act separately on horizontal and vertical tangents as double transpositions. More precisely, for $k \in \{0,1,\lambda\}$, 
\begin{itemize}
  \item $\tsigmak (H_{\infty}) = H_k$ and $\tsigmak (V_{\infty}) = V_k$.
  \item $\tsigmak (H_i) = H_j$ and $\tsigmak (V_i) = V_j$ for $\{i,j,k\} = \{0,1,\lambda\}$.
  \item $\phi_T(H_i) = V_i$.
\end{itemize}

  Let $\Aut(\PP, \Gamma)$ be the group of automorphisms of $\PP$ leaving invariant $\Gamma$. Then, we have:

\begin{proposition} \label{prop:autPP}
  The group $\Aut(\PP, \Gamma)$ is isomorphic to $(\mathbb{Z}/2\mathbb{Z})^3$, and it is generated by the involutions $\tsigmak$ and $\phi_T$. 
\end{proposition}  
\begin{proof}
  The group $\Aut(\PP)$ is isomorphic to $\text{PGL}(2) \times \text{PGL}(2)$. Let $\gamma$ be an automorphism preserving $\Gamma$. Then, $\gamma$ must preserve the set $T$ of vertical and horizontal tangents to $\Gamma$. Also, $\gamma$ is completely determined by its action on this set. The subgroup of $\Aut(\PP, \Gamma)$ of elements preserving verticals is generated by $\{\tsigmak\}$, and the group of automorphisms preserving $T$ is generated by $\{\tsigmak, \phi_T\}$. 
\end{proof}
  Each involution $\tsigmak$ lifts to two automorphisms $\sigma_k$ and $\sigma_k \circ \tau$ of $S$. These involutions act transitively on the set $\Omega$. Choose $\sigma_k$ to be the lift of $\tsigmak$ such that $\sigma_k(\Pi_{t\infty}) = \Pi_k$ (and thus $\sigma_k \circ \tau (\Pi_{t\infty})=\Pi_{tk}$, and $\sigma_{\infty} = \tau$). Finally, let $\psi_T$ be the lift of $\phi_T$ such that $\psi_T (\Pi_{t\infty}) = C$.  The action of these involutions on the set $\Omega$ is summarized in Table \ref{table:actionomega}.

\begin{table}[htp]
\begin{center}
\begin{tabular}{|c|llll|}
\hline
&
$ \Pi_t $ & $\longleftrightarrow $ & $ \Pi $ 
&
\\ $\tau$ & 
$\Pi_i $ & $\longleftrightarrow $ & $ \Pi_{it}$ & $\text{for } \{i \not = t \}$
\\ & 
$\Pi_{i\infty} $ & $\longleftrightarrow $ & $ \Pi_{jk}$ & $\text{for } \{i,j,k\} = \{0,1,\lambda\}$
\\ \hline
 & 
$\Pi_k $ & $\longleftrightarrow $ & $ \Pi_{t \infty}$ &
\\ &
$\Pi_{kt} $ & $\longleftrightarrow $ & $ \Pi_{\infty}$
&
\\   $\sigma_k$ &
$\Pi_{i\infty} $ & $\longleftrightarrow $ & $ \Pi_{jt}$ & $\text{for } \{i,j,k\} = \{0,1,\lambda\}$ 
\\ 
$(k \not = \infty)$ & $\Pi_{ik} $ & $\longleftrightarrow $ & $ \Pi_{i} $ &
\\ &  
$\Pi $ & $\longleftrightarrow $ & $ \Pi_{ij}$ &
\\ \hline
&
$ \Pi_{t\infty} $ & $\longleftrightarrow $ & $ \Pi $ 
&
\\ $\psi_T$ & 
$\Pi_{t} $ & $\longleftrightarrow $ & $ \Pi_{\infty} $ & 
\\ & 
$\Pi_{ij} $ & $\longleftrightarrow $ & $ \Pi_{k}$ & $\text{for } \{i,j,k\} = \{0,1,\lambda\}$
\\ \hline
\end{tabular}
\end{center}
\caption{The action of the involutions $\tau$, $\sigma_k$ and $\psi_T$ on the set $\Omega$.}
\label{table:actionomega}
\end{table}

  By Proposition \ref{prop:autPP}, the subgroup of $\AutS$ generated by $\sigma_0$, $\sigma_1$ and $\psi_T$ is isomorphic to $\Aut(\PP, \Gamma)$. Since $\tau$ is not an automorphism of $\Aut(\PP, \Gamma)$, we directly obtain by cardinality of $\AutS$ that
\begin{align*}
  \AutS = \langle \sigma_0 , \sigma_1, \psi_T, \tau \rangle
\end{align*}
  This completes the proof of Theorem \ref{th:autS}.

\bibliography{mibib}

\begin{thebibliography}{10}

\bibitem{atiyahvectorbundles}
M.~F. Atiyah.
\newblock Vector bundles over an elliptic curve.
\newblock {\em Proc. London Math. Soc. (3)}, 7:414--452, 1957.

\bibitem{biswashollakumar}
I.~Biswas, Y.~I. Holla, and C.~Kumar.
\newblock On moduli spaces of parabolic vector bundles of rank 2 over {$\mathbb
  C\mathbb P^1$}.
\newblock {\em Michigan Math. J.}, 59(2):467--475, 2010.

\bibitem{bodenhu}
H.~U. Boden and Y.~Hu.
\newblock Variations of moduli of parabolic bundles.
\newblock {\em Math. Ann.}, 301(3):539--559, 1995.

\bibitem{bolognesiweddle}
M.~Bolognesi.
\newblock On {W}eddle surfaces and their moduli.
\newblock {\em Adv. Geom.}, 7(1):113--144, 2007.

\bibitem{bolognesikummer}
M.~Bolognesi.
\newblock A conic bundle degenerating on the {K}ummer surface.
\newblock {\em Math. Z.}, 261(1):149--168, 2009.

\bibitem{dolgachev}
I.~V. Dolgachev.
\newblock {\em Classical algebraic geometry}.
\newblock Cambridge University Press, Cambridge, 2012.

\bibitem{esnault}
H.~Esnault and C.~Hertling.
\newblock Semistable bundles on curves and reducible representations of the
  fundamental group.
\newblock {\em Internat. J. Math.}, 12(7):847--855, 2001.

\bibitem{heuloray}
V.~Heu and F.~Loray.
\newblock {Flat rank 2 vector bundles on genus 2 curves}.
\newblock {\em Memoirs of the AMS (to appear)}, January 2014.

\bibitem{kumar}
C.~Kumar.
\newblock Invariant vector bundles of rank 2 on hyperelliptic curves.
\newblock {\em Michigan Math. J.}, 47(3):575--584, 2000.

\bibitem{lorayisom}
F.~Loray.
\newblock Isomonodromic deformations of {L}am\'e connections, the {P}ainlev\'e
  {VI} equation and {O}kamoto symmetry.
\newblock {\em Izv. Ross. Akad. Nauk Ser. Mat.}, 80(1):119--176, 2016.

\bibitem{loraysaito}
F.~Loray and M.-H. Saito.
\newblock {Lagrangian fibration in duality on moduli space of rank two
  logarithmic connections over the projective line}.
\newblock {\em {I. M. R. N.}}, 4:995--1043, 2015.

\bibitem{machu}
F.-X. Machu.
\newblock Monodromy of a class of logarithmic connections on an elliptic curve.
\newblock {\em SIGMA Symmetry Integrability Geom. Methods Appl.}, 3:Paper 082,
  31, 2007.

\bibitem{maruyamayokogawa}
M.~Maruyama and K.~Yokogawa.
\newblock Moduli of parabolic stable sheaves.
\newblock {\em Math. Ann.}, 293(1):77--99, 1992.

\bibitem{mehtaseshadri}
V.~B. Mehta and C.~S. Seshadri.
\newblock Moduli of vector bundles on curves with parabolic structures.
\newblock {\em Math. Ann.}, 248(3):205--239, 1980.

\bibitem{seshadri}
C.~S. Seshadri.
\newblock Moduli of vector bundles on curves with parabolic structures.
\newblock {\em Bull. Amer. Math. Soc.}, 83(1):124--126, 1977.

\bibitem{tu}
L.~W. Tu.
\newblock Semistable bundles over an elliptic curve.
\newblock {\em Adv. Math.}, 98(1):1--26, 1993.

\end{thebibliography}
\bibliographystyle{plain}
  
\end{document}